\documentclass[11pt, reqno]{amsart}
\usepackage{amsfonts,amsmath,amssymb,amsthm,epsfig,cite,graphicx,hyperref,
	color, esint,fancyhdr, enumerate, latexsym,amsrefs, makecell}
\usepackage{color}
\usepackage[mathscr]{eucal}
\usepackage{comment}
\usepackage[nameinlink]{cleveref}

\numberwithin{equation}{section}

\theoremstyle{definition}
\newtheorem{definition}{Definition}[section]
\newtheorem{example}[definition]{Example}

\theoremstyle{remark}
\newtheorem{remark}[definition]{Remark}
\theoremstyle{plain}
\newtheorem{theorem}[definition]{Theorem}
\newtheorem{lemma}[definition]{Lemma}
\newtheorem{proposition}[definition]{Proposition}

\newtheorem{result}[definition]{Result}

\newtheorem{corollary}[definition]{Corollary}

\newtheorem{conjecture}[definition]{Conjecture}

\newcommand{\eps}{\varepsilon}

\newcommand{\al}{\alpha}

\newcommand{\natu}{\mathbb{N}}
\newcommand{\vphi}{\varphi}

\newcommand{\om}{\omega}
\newcommand{\Om}{\Omega}

\newcommand{\bdy}{\partial}

\newcommand{\smoo}{\mathcal{C}}
\newcommand{\hol}{\mathcal{O}}

\newcommand{\rl}{{\sf Re}}

\newcommand{\impl}{\Longrightarrow}

\newcommand\wtil[1]{\widetilde{#1}}
\newcommand{\CC}{\mathbb{C}^2}
\newcommand{\cplx}{\mathbb{C}}

\newcommand{\rea}{\mathbb{R}}
\newcommand{\cn}{\mathbb{C}^n}
\begin{document}
	\title[Approximations on certain domains in $\mathbb{C}^n$ ]{Approximations on certain domains of $\mathbb{C}^{n}$}
	\author{Sanjoy Chatterjee and Sushil Gorai}
	\address{Indian Statistical Institute, Kolkata}
	\email{ramvivsar@gmail.com}
	
	\address{Department of Mathematics and Statistics, Indian Institute of Science Education and Research Kolkata,
		Mohanpur -- 741 246}
	\email{sushil.gorai@iiserkol.ac.in, sushil.gorai@gmail.com}
	%\thanks{Sanjoy Chatterjee is supported by CSIR fellowship (File No-09/921(0283)/2019-EMR-I). Sushil Gorai is partially supported by a Matrics Research Grant (MTR/2017/000974) and a Core Research Grant (CRG/2022/003560) of SERB}
	\keywords{Approximation of biholomorphism,  Automorphisms of $\mathbb{C}^n$, Runge domains, positive time flow invariant domains, Loewner chain}
	\subjclass[2020]{Primary: 32M17, 32E30, 30C45 }

	\date{\today}

	%\begin{comment}

	\begin{abstract}
		In this paper, we study the domains in $\mathbb{C}^n$ that are invariant under the positive flows of some globally defined, complete holomorphic vector field 
  with a globally attracting fixed point at the origin. Our first result says that such a domain $\Omega$ is always Runge. Next, with an additional assumption on the rate of convergence of the flow, we show that 
   any biholomorphism $\Phi\colon \Omega \to \Phi(\Omega)$, with $\Phi(\Omega)$ is Runge, can be approximated by automorphisms of $\mathbb{C}^{n}$ uniformly on compacts. This generalizes all earlier known theorems in this direction substantially, even when the vector field is linear.  As an application of our approximation results, on such domains that are also complete hyperbolic, we show that any Loewner PDE in a complete hyperbolic domain $\Omega$ admits an essentially unique univalent solution with values in $\mathbb{C}^n$. We also provide  an approximation result for volume preserving  biholomorphisms on above domains. 
   We provide several examples of such domains.
	\end{abstract}	
	%\end{comment}

	\maketitle
	
	%\tableofcontents
	
	\section{Introduction}\label{Intro}

A domain $\Om  \subseteq \cn$ is said to be a Runge domain if every holomorphic function $f: \Om \to \mathbb{C}$ can be approximated by holomorphic polynomials uniformly on every compact subset of $\Om$. Runge domains are fundamental in the study of function theory in several complex variables. In one variable, they are the simply connected domains in $\cplx$, hence, a topological property classifies them. In higher dimensions, no such characterization exists. Characterizing the Runge domain in higher dimensions is a very difficult problem. In this paper,  we give a class of domains that are Runge.  A {\em holomorphic vector field} $V$ on a domain  $\Om \subset \mathbb{C}^{n}$ is a real vector field on $\Om$  such that  $$V(z)=\sum_{j=1}^{n} a_{j}(z) \frac{\partial }{\partial  x_{j}}+b_{j}(z) \frac{\partial }{\partial  y_{j}}$$ such that $(a_j+ib_{j})$ is holomorphic function on $\Om$  for all $j \in \{1,2, \cdots ,n\}$,  $z_{j}=x_{j}+iy_{j}$ are coordinate of $\mathbb{C}^{n}$. We denote the set of all holomorphic vector fields on $\cn$ by $\mathfrak{X}_{\mathcal{O}}(\cn)$. 
We now state our first theorem.
\begin{theorem}\label{thm-spirallikeRunge1}
		Let $F$ be a complete holomorphic vector field with a globally attracting fixed point at the origin. Assume that $\Om $ is domain which is invariant under positive time flows of $F$ and contains the origin. Then $\Omega$ is a Runge domain.
	\end{theorem}

\noindent It is known that any star shaped domain in $\cn$ is Runge (see \cite{Kasimi}). Any star shaped domain with respect to the origin is invariant under the positive time flows of the vector field $-I$. In \cite[Theorem-3.1]{Hamada15}, Hamada proved the same when the domain is invariant under positive time flows of certain linear vector fields. Hamada calls those domains as spirallike domains. A detailed discussion about such domains are given in \Cref{sec-prelims}.
\medskip

	The main question that we study in this paper is: {\em For what classes of domains $\Omega$ any biholomorphic map from $\Omega$ into $\cn$ can be approximated by automorphisms of $\cn$?} 
  Thanks to the Anders\'{e}n-Lempert theory, a positive answer of this question will lead to approximating any biholomorphic map by simpler classes of maps, namely the finite composition of shears and overshears. 
 In this paper, we give a large class of domains for which such approximation holds. 
 To proceed further in our discussion, a brief description of the automorphism group of $\cn$ will be useful.
 Let Aut$(\cn)$ be set of all holomorphic mappings $f:\cn \to \cn$ such that $f^{-1}$ exist and also holomorphic. For $n=1$, we know that Aut$(\mathbb{C})$ consists only of affine mappings $z\mapsto az+b,(a,b \in \mathbb{C},\; a\neq 0)$. However, Aut$(\cn)$ is huge for $n \geq 2$. Writing the standard coordinates as $(z_{1},z_{2}, \cdots z_{n})=(z',z_{n})$ on $\cn$ we see that Aut$(\cn)$ contains mappings of the form
	\begin{equation}\label{eq-shear}
		(z',z_n) \mapsto (z', z_{n}+f(z'))
	\end{equation}
	\begin{equation}\label{eq-overshear}
		(z',z_n) \mapsto (z', e^{f(z')}z_{n}), 
	\end{equation}
	where $f:\mathbb{C}^{n-1} \to \mathbb{C}$ is any holomorphic mapping. The inverse of the above mapping is obtained by replacing $f$ with $-f$ in \eqref{eq-shear} and \eqref{eq-overshear}. The mappings of forms \eqref{eq-shear} and their $SL(n,\mathbb{C})$ conjugates are called {\em shears or additive shears}, and the mappings \eqref{eq-overshear} and their $GL(n, \mathbb{C})$ conjugates are called {\em overshears or multiplicative shears}. The subgroup generated by elements of overshears and shears is denoted by $S(n)$, $S_{1}(n)$, respectively. The set of all volume preserving automorphisms of $\cn $ is denoted by $\text{Aut}_{1}(\cn)$ and is defined by $\text{Aut}_{1}(\cn):=\{F\in \text{Aut}(\cn): \text{det}(DF(z)) \equiv 1\}$. Our next result is about approximating biholomorphic maps on certain domains in $\cn$ which are invariant under the positive time flow of certain holomorphic vector field  by compositions of shears and overshears, special members of the $\text{Aut}(\cn)$. Before starting our main result let us survey a few things from the existing literature.
The  study of Aut$(\cn)$ was initiated by Rosay and Rudin \cite{RosayRudin}. Questions about the density of the compositions of shears in the class of all volume preserving automorphisms of $\cn$  and the density of compositions of
	overshears in Aut$(\cn)$ were asked in \cite{RosayRudin}, which were answered positively by Anders\'{e}n \cite{Andersen} and by Anders\'{e}n and Lempert \cite{AndersenLemp} respectively, giving birth of the now-famous Anders\'{e}n-Lempert theory \cites{Andersen, AndersenLemp} (see \cite{Forstbook} for more details).
	It is a natural question to ask for characterization of the domains $\Omega\subset\cn$ for which any biholomorphism from $\Om$ into $\cn$ can be approximated 
	uniformly on compacts of $\Om$ by automorphisms of $\cplx^n$, i.e. in view of Anders\'{e}n-Lempert theory by composition of Overshears. Anders\'{e}n and Lempert \cite{AndersenLemp} considered this question and proved:

 \begin{result}\cite[Theorem 2.1]{AndersenLemp}\label{thm-AndersenLemp}
    Let $\Om$ be a starshaped domain $\cplx^n$ and $\Phi:\Om\to\cn$ be a biholomorphic map such that $\Phi(\Om)$ is Runge. Then $\Phi$ can be approximated uniformly on compact subsets by compositions of overshears. If $\det D\Phi\equiv 1$, then $\Phi$ can be approximated uniformly on compact subsets by the composition of shears.
 \end{result}
 \noindent
 Note that, if a biholomorphism $\Phi:\Om\to\Phi(\Om)$, with $\Phi(\Om)$ Runge, can be approximated uniformly on compacts by elements of Aut$({\cn})$, then in view of the density theorem due to  Anders\'{e}n and Lempert \cite[Theorem~1.3]{AndersenLemp}, we get that the biholomorphism $\Phi$ can be approximated by the composition of overshears.
	
In a seminal paper Forstneri\v{c} and Rosay \cite[Theorem-1.1]{FR1993}, \cite[Theorem-1.1]{Forst1994} gave a sufficient condition for Runge domain under which any biholomorphic map can be approximated by automorphisms of $\cn$.

 \begin{result}[Forstneri\v{c}-Rosay]\label{res-Forstneric-Rosay}
Let $\Om \subset \mathbb{C}^{n}$ be a Runge domain. Assume that $H: [0,1] \times \Om \to \mathbb{C}^{n}$ is a mapping of class $\smoo^{p}(0 \leq p < \infty)$ such that for each $t \in [0,1]$, $H_{t}=H(t, \cdot)\colon \Om \to \mathbb{C}^{n}$ is a biholomorphic mapping onto a Runge domain $\Om_{t}:=H_{t}(\Om) \subset \mathbb{C}^{n}$ and the map $H_{0}$ can be approximated uniformly on every compact subset of $\Om$ by automorphisms of $\mathbb{C}^{n}$. Then 

\begin{itemize}
    \item [(i)]
  for each compact subset $K \subset \Om$ and each $\epsilon>0$ there exist a smooth map $\Psi: [0,1] \times \cn \to \mathbb{C}^{n} $ 
  such that  $\Psi_{t}:=\Psi(t,\cdot) \in \text{Aut}(\cn)$, for every $t \in [0,1]$ 
  and $\|H-\Psi\|_{\smoo^{p}([0,1] \times K)}<\epsilon$,
  \item [(ii)]
  if  we have $H_{0} \in Aut(\cn)$, then we can choose $\Psi$ such that $\Psi_{0}=H_{0}$,
  \item 
  [(iii)]
  if  $\Om$ is pseudoconvex, $H^{(n-1)}(\Om, \cplx)=0$, $\text{det}DH_{t}(z) \equiv 1$ on $\Om$ and  $H_{0}$ can be approximated on $\Om$ by elements of $\text{Aut}_{1}(\cn)$,
 then every $H_{t}$ can be approximated on $\Om$ by elements of $\text{Aut}_{1}(\cn)$.
\end{itemize}
\end{result}
\noindent
This result by Forstneri\v{c} and Rosay is the most general result in this direction. The smoothness assumption of the family is lowered in \cite{Forst1994} and \cite{FKu}. But, for a given domain,
constructing such an isotopy $\{\Phi_t\}$ in \Cref{res-Forstneric-Rosay} is often difficult.
Therefore, it is interesting to know the classes of domains to which the above result can be applied.
Hamada \cite{Hamada15} gave a class of domains in the following result.
	\begin{result} \cite[Theorem 4.2]{Hamada15}\label{thm-hamada2}
		Let $\Om$ be a domain containing the origin that satisfies the following conditions: 
		$e^{-tA}w\in\Om$ for all $w\in\Om$, for all $t \geq 0$ where $A\in M_n(\cplx)$ such that $\inf_{||z||=1}\rl\langle Az, z\rangle>0$. 
		If $\max_{\lambda\in\sigma(A)}\rl(\lambda)<2 \inf_{||z||=1}\rl\langle Az, z\rangle$, where $\sigma(A)$ is the spectrum of $A$, then any biholomorphism $\phi:\Om\to\cn$ with $\phi(\Om)$ Runge, can be approximated uniformly on compacts of $\Om$ by automorphisms of $\cn$.
	\end{result} 
 In this paper, we focus on non-linear complete holomorphic vector fields with a globally attracting fixed point at the origin (see \Cref{sec-prelims} and the domains that remains invariant under positive time flows of such vector fields. We show under a suitable growth condition on the flows the approximation result holds.
To state it precisely, we need the following: 

Let $\sigma(A)$ denote the spectrum of the matrix $A$. Let $k_{+}(A):=\max \{\rl \lambda : \lambda \in \sigma (A)\}$ and $ k_{-}(A):=\min\{\rl \lambda : \lambda \in \sigma (A)\}$. We now present the main result of this article:
\begin{theorem}\label{thm-apprxges}
		Let $V \in \mathfrak{X}_{\mathcal{O}}{(\mathbb{C}^{n})}$, $n \geq 2$, be a complete holomorphic  vector field with a globally attracting fixed point at the origin and $DV(0)=A$. Assume that $ \Om  \subseteq \cn$ contains the origin and is invariant under the positive time flows of $V$. Suppose that  there exist $\gamma, \alpha>0$ such that 
  \begin{itemize}
 \item [(i)] $\|X(t,z)\|\leq \gamma e^{-\alpha t}\|z\|$ for all $t \geq 0$ and for all $z \in \Om$, where $X(t,z)$ is the flow of  the vector field $V$.
 \item [(ii)] \begin{align}\label{C:conditi1}
2\alpha+k_{-}(A)&>0.
\end{align}	
\end{itemize}
 Then any biholomorphism $\Phi: \Om \to \Phi(\Om) \subseteq \cn$ with $\Phi(\Om)$ Runge can be approximated by automorphisms of $\mathbb{C}^{n}$  uniformly on every compact subset of $\Omega$.
 \end{theorem}
\begin{remark}
    Conditions $(i)$ and $(ii)$ of \Cref{thm-apprxges} gives a uniform exponential decay which, in turn, provides the asymptotic behavior of the flows of $V$ as  exponential type up to a biholomorphism of $\Om$ (see \Cref{L:composconver}). This plays a vital role in our proof.
\end{remark}
  We obtain the following corollary from \Cref{thm-apprxges}, restricting the vector field to be linear which already gives a substantial generalization to \Cref{thm-hamada2}.

	\begin{corollary}\label{thm-linear case}
	Let $A\in GL(n,\cplx)$ with $2k_{+}(A)<k_{-}(A)$ and $ \Om \subset \cn$ be a domain containing the origin and  $P^{-1}e^{tA}Pw\in\Om$ for all $w\in\Om$ for some $P\in GL(n,\cplx)$. Then any biholomorphism  $\Phi \colon \Om \to \Phi(\Om)$ with $\Phi(\Om)$ Runge can be approximated by elements of $Aut(\cn)$ uniformly on every compact subset of $\Om$.
	 \end{corollary}

We obtain the following corollary of the \Cref{thm-apprxges}.
\begin{corollary}\label{thm-pro}
Let $\Om_{1}\subset \cn$, $\Om_{2} \subset \mathbb{C}^{m} $ be two domains containing the origin and are invariant under positive time flow of complete vector fields  $V_{1} \in \mathfrak{X}_{\mathcal{O}}(\cn)$ and $V_{2} \in \mathfrak{X}_{\mathcal{O}}(\mathbb{C}^{m})$ respectively. Both the vector fields has a globally attracting fixed point at the origin.
Suppose that $X_{1}(t,z)$, $X_{2}(t,w)$ denotes the flow of the vector field $V_{1}$, $V_{2}$. Assume that there exists $\gamma_{1}, \gamma_{2}, \alpha_{1}, \alpha_{2}>0$ such that $\|X_{j}(t,z_{j})\|\leq \gamma_{j}e^{-\alpha_{j}t}\|z_{j}\|$, for all $t\geq 0$, for all $z_{j} \in \Omega_{j}$ for $j \in \{1,2\}$.  Suppose that $$\min\{\alpha_{1}, \alpha_{2}\}>\max\{\frac{-k_{-}(A)}{2}, \frac{-k_{-}(B)}{2}\}.$$ Then, any biholomorphism $\Phi \colon \Om_{1} \times \Om_{2} \to \Phi(\Om_{1}\times \Om_{2}) \subseteq \mathbb{C}^{n} \times \mathbb{C}^{m}$ with $\Phi(\Om_{1}\times \Om_{2})$ Runge can be approximated by elements of $\text{Aut} ( \mathbb{C}^{m+n})$. 
	\end{corollary}

\noindent Here we pose a conjecture
\begin{conjecture}
Conclusion of \Cref{thm-apprxges} holds true without condition (ii).
\end{conjecture}

 A  holomorphic symplectic  form $\omega$ on $\mathbb{C}^{2n}$ is closed holomorphic $2$-form on $\mathbb{C}^{2n}$ whose highest exterior power $\omega^{n}$ is nowhere vanishing. The standard holomorphic symplectic form on $\mathbb{C}^{2n}$ with coordinates $z_{1}, \cdots, z_{n},~w_{1}, \cdots   ,w_{n}$ is $ \omega= \sum_{j=1}^{n}\,dz_{i}\wedge \,dw_{i}$
	Let $\Om \subset\cn$ be a domain. A holomorphic map  $f: \Om \to  \cplx^{2n}$ is called symplectic if $f^{*}(\om)=\om$. A holomorphic vector field $V$ on $\Om$ is said to be symplectic if $\,d\iota_{V}(\om)=0$. An automorphism $F:\cplx^{2n} \to \cplx^{2n} $ is symplectic if $F^{*}(\om)=\om$. We denote the set of all symplectic holomorphism by $Aut_{\om}(\cn)$.
 
  We now present an approximation result where a
 volume preserving biholomorphisms and symplectic biholomorphism that are defined on a domain that is invariant under positive time flows of some vector field,  by volume preserving automorphism and symplectic automorphisms of $\cn$ respectively. 
 
 \begin{theorem}\label{thm-Volsympapprox}
 Let $V \in \mathfrak{X}_{\mathcal{O}}{(\mathbb{C}^{n})}$, $n \geq 2$, be a complete  vector field with a globally attracting fixed point at the origin and $DV(0)=A$. Assume that $ \Om \subseteq \mathbb{C}^{2n}$ is a pseudoconvex domain containing the origin and invariant under positive time flows of $V$. Suppose that  there exist $\gamma, \alpha>0$ such that $\|X(t,z)\|\leq \gamma e^{-\alpha t}\|z\|$ for all $t \geq 0$ and for all $z \in \Om$, where $X(t,z)$ is the flow of the vector field $V$. Assume that
$2\alpha+k_{-}(A)>0.$
 Let  $\Phi: \Om \to \Phi(\Om)$ be a biholomorphism  with $\Phi(\Om)$ Runge. Then the following holds:
 \begin{itemize}
 \item[(i)]
 If  $div_{\om^{n}}V(z)$ is non zero constant and  $\text{det}(D\Phi(z)) \equiv 1$, where $\om^{n}$ is  standard holomorphic $(n,0)$ form $\,dz_{1} \wedge \,dz_{2} \wedge \cdots \wedge \,dz_{n}$,  then $\Phi$ can be approximated by elements of $Aut_{1}(\cn)$  uniformly on every compact subset of  $\Om$.
 \item [(ii)]
 Let $n=2m$ and $(z_{1}, z_{2}, \ldots ,z_{m}, w_{1}, w_{2}, \ldots ,w_{m})$ be a coordinate of $\mathbb{C}^{n}$ and $\omega=\sum_{j=1}^{m}\,dz_{i} \wedge \,dw_{i}$ is a symplectic form on $\cn$. If $V$ is symplectic vector field and $\Phi^{*}(\omega)=\omega$, then $\Phi$ can be approximated by elements of $Aut_{\omega}(\cn)$ locally uniformly on $\Om$.
 \end{itemize}
  \end{theorem}   
\medskip
	
	Our next result is an application of Theorem~\ref{thm-apprxges} to the theory of Loewner PDE. The study was motivated by Arosio-Bracci-Wold \cite{ABFW13} and 
	Hamada \cite{Hamada15}. We present a brief introduction to Loewner PDE before presenting our result. The Loewner ordinary differential equation in the unit disc was introduced by Loewner\cite{Loew1923} in 1923 to study Bieberbach's conjecture about the coefficients of a univalent map on the unit disc. The theory in one variable was developed later by Kufarev \cite{Kuf} and Pommerenke \cite{Pomm1965} (see \cite{MBCD2010} for a brief survey of Loewner differential equations). 
	The Loewner theory has been used to prove many deep results in geometric function theory \cite{GKbook}.
	It was one of the main ingredients in de Brange's proof of Bieberbach conjecture \cite{dB1985}.
	Then there were several generalizations;  Some of them are:  chordal Loewner theory \cite{KSS1968} and Loewner theory in several complex variables \cites{Pfal1974, Pfal1975}. 
	The theory of the Loewner equation in several complex variables was extended by 
 Pfaltzgraff \cites{Pfal1974, Pfal1975}, and later studied by many others \cites{Poreda1991,GHK2002, GHKK2008, HK2003, GKP2005}. 
	A map $H:\Om\times[0,\infty)\to\cplx^n$ is called a Herglotz vector field of order $d$ on $\Om$ if 
	\begin{itemize}
		\item [(i)]For each $z\in\Om$, $H(z,\cdot):[0,\infty)\to\Om$ is measurable.
		\item [(ii)]  For each $s\in[0,\infty)$, $H(.,s)$ is holomorphic vector field and $H(\cdot , s)$ is a semicomplete holomorphic vector field for almost all $s \in [0, \infty)$, i.e. for each $z_0\in\Om$, the ODE defined by
		$$
		\dfrac{dw}{dt}=H(w(t),s),\;\;\; w(0)=z_0
		$$
		has a solution for all $t\in[0,\infty)$ for almost all $s \in [0, \infty)     $.
		\item[(iii)] 
  For each $t_0\in[0,\infty)$, and for each compact $K\subset\Om$, there exists a function $c_{t_{0}}^K\in L^d([0,t_0],[0,\infty))$ such that $||H(z,t)||\leq c_{t_{0}}^K$ for all $z\in K$ almost all $t\in[0, t_{0}]$. 
	\end{itemize}
	For a Herglotz vector field $H$ of order $d\in[1,\infty]$ on a complete hyperbolic domain $\Om$ the Loewner PDE is the following equation: 
	\begin{align}
		\dfrac{\partial f_t(z)}{\partial t} &=-df_t(z)H(z,t).\label{loewnerpde}
	\end{align} 
	A solution of the Loewner PDE is a family of holomorphic maps $\{f_t\}_{t\in[0,\infty)}$ from $\Om$ to a complex manifold $M$ which is locally absolutely continuous on $[0,\infty)$ satisfies \eqref{loewnerpde} and $t\mapsto f_t$ is continuous in the topology of holomorphic maps from $\Om$ to $M$. Such a family of maps is closely related to evolution family $\{\varphi_{s,t}\}$ that satisfies Loewner ODE \eqref{loewnerode}
	
	\begin{align}\label{loewnerode}
		\dfrac{\partial \varphi_{s,t}(z)}{\partial t} &=H(\varphi_{s,t}(z),t),~~a.e ~~t \in [s,\infty)
	\end{align} 
	in the following sense: 
	\begin{align}
		f_s &=f_t\circ\varphi_{s,t}, \;\; 0\leq s\leq t. \label{funleqn}
	\end{align}
	The main open problem in the field is; {\em To solve whether, on a complex complete hyperbolic domain, there is a solution to Loewner PDE with values in $\cplx^n$}. Solution of such Loewner PDE with values in certain abstract complex manifold has already been found in \cite{BCD2009}.
	Recently the connection between the approximation of biholomorphic maps and the solution of Loewner PDE  has been found \cite{ABFW13}. As an application of Theorem~\ref{thm-apprxges} we state the following.

	\begin{theorem}\label{T:Loewner Pde}
		Let $n \geq 2$. Let $\Omega \subset \mathbb{C}^{n}$ be a complete hyperbolic domain containing the origin and invariant under positive time flows of the vector field $V \in\mathfrak{X}_{\mathcal{O}}{(\mathbb{C}^{n})}$ with a globally attracting fixed point at the origin and $DV(0)=A$. Suppose that  there exist $\gamma, \alpha>0$ such that $\|X(t,z)\|\leq \gamma e^{-\alpha t}\|z\|$, for all $t \geq 0$ and for all $z \in \Om$, where $X(t,z)$ is the flow of  the vector field $V$. Assume that 
		\begin{align}\label{E:conditi2}
		2\alpha+k_{-}(A)&>0.	
		\end{align}
		\noindent
		 Let $G: \Omega \times \mathbb{R}^{+} \to \mathbb{C}^{n}$ be a Harglotz vector field of order $d \in [1, \infty]$.
		Then there exist a Loewner chain $f_{t}: \Omega \to \mathbb{C}^{n}$ of order $d$
		which solves the Loewner PDE 
		\begin{align}\label{E:llpd}
			\frac{\partial f_{t}}{\partial t}(z)&=-df_{t}(z)(G(z,t)) ~~a.e.~~ t \geq 0~~\forall z \in \Omega. 
		\end{align}	
		Moreover, $R=\cup_{t \geq 0}f_{t}(\Omega)$ is a Runge and Stein domain in $\mathbb{C}^{n}$ and any other solution to \eqref{E:llpd} with values in $\cn$ is of the form $(\vphi \circ f_{t})$ for a suitable map $\vphi: R \to \mathbb{C}^{n}$.
	\end{theorem} 
 
\noindent Theorem~\ref{T:Loewner Pde} 
 gives a partial answer to the open question  whether any Loewner PDE in a complete hyperbolic domain in $\cn$ admits a univalent solution in $\cn$. This generalizes results of \cite{ABFW13, Hamada15} Other partial positive answers are \cite{ABFW13,Hamada15, Duren}. 

 In the final section of the paper, we provide several examples of domains that satisfy the conditions of our theorems. Thus, those domains allow to approximate any biholomorphic maps onto a Runge domain by automorphisms of $\cn$. 
	we also provide an example of a complete hyperbolic Hartogs domain which is invariant with respect to  positive time flow of a nonlinear holomorphic  vector field  but not with respect to  any linear vector field. 
 
	\begin{comment}
	About the layout of the paper: In \Cref{sec-prelims}, we provide basic definitions and collect some results from the literature that we will use in our proofs. In \Cref{sec-Runge}, a proof of \Cref{thm-spirallikeRunge1} is provided.
  \Cref{sec-approx}  is the heart of this paper.  The proof of \Cref{thm-apprxges}, \Cref{thm-Volsympapprox} and \Cref{thm-pro} are presented there. The proof of \Cref{T:Loewner Pde} is presented in \Cref{sec-Application}. \Cref{sec-example} we provide some examples of a spirallike domain that satisfies the conditions of our theorems.
	\end{comment}
	
	% SECTION PRELIMINARIES

	\section{Technical Preliminaries}\label{sec-prelims}
	
	In this section, we first discuss about the holomorphic vector fields and their flows. Then we make a comparison of domains that were considered by Hamada in \cite{Hamada15} with ours when the vector field is linear. 
Let $\Om \subset \cn$ and $V$ be a holomorphic vector field on $\Om$. 
 To find the real flow of $V$ we consider the following equation 
 \begin{align}\label{E:diffeq}
		\dot{x}(t)&=V(x(t)),\;\;\;
		x(0)=x_{0},
	\end{align}
 We  consider the holomorphic differential
	equation associated to \eqref{E:diffeq} defined by  
	\begin{align}\label{E:diffeq1}
		\dot{x}(z)&=\widetilde{V}(x(z))\notag\\
		x(0)&=z_{0},
	\end{align}
	where $\widetilde{V}(z):=\frac{1}{2}(V(z)-iJV(z))$ , $J$ is the complex structure given by
 $J(\frac{\partial}{\partial x_{j}})=\frac{\partial}{\partial y_{j}}$, $J(\frac{\partial}{\partial y_{j}})=-\frac{\partial}{\partial x_{j}}$ for holomorphic coordinates $z_j=x_j+iy_j, j=1,\dots, n$,  
 and  $\dot{x}(z)$  denotes the complex derivative of $x$ with respect to the complex variable $z$. Again, the $(\ref{E:diffeq1})$ has a unique local holomorphic solution $z \to X(z, z_{0})$, which depends  holomorphically on  both $z$ and $z_{0}$. Using Cauchy-Riemann equation we conclude that (\ref{E:diffeq1}) is equivalent to  the system of equations 
	
	$$\frac{d\phi}{dt}=V(\phi(t)); ~~	\frac{d\psi}{ds}=i\cdot V(\psi(s)),$$ 
	where $z=t+is$,    $X(0)=z_{0}$. Thus, we have  $X(t+is, z)=\phi_{t}(\psi_{s}(z_{0}))$.
	
		A vector field $V$ on $\Om$ is said to be $\mathbb{R}$-complete if solution of (\ref{E:diffeq}) exists for $t \in \mathbb{R}$ for every $x_{0} \in \Om$. If $V$ is holomorphic vector field and solution of  $(\ref{E:diffeq1})$ exists for all $z \in \mathbb{C}$ for all $z_{0} \in \Om$ then $V$ is said to be $\mathbb{C}$-complete.
From \cite[Corollary 2.2]{Forst1996} we get that every $\mathbb{R}$-complete  holomorphic vector field is $\mathbb{C}$-complete. We will use this crucially in our proofs.
We always mean $\cplx$-complete vector field whenever we say a vector field is complete.	
If a holomorphic vector field $V \in \mathfrak{X}_{\mathcal{O}}(\cn)$ is  $\mathbb{C}$-complete) then its associated flow defines a smooth action of $(\mathbb{C},+)$ on $\cn$ through holomorphic automorphisms of $\cn$.

	In this paper we will consider the holomorphic vector fields with the following properties: 
	\begin{itemize}
		\item[i)] The vector field is defined on the whole of $\cn$.
		\item[ii)] The vector field is $\rea$-complete.
		\item[iv)] The vector field has a globally attracting fixed point.
	\end{itemize}
	Since the above properties play very important roles in our results, we just mention those here briefly before going further in our exposition. 
	We consider the following system of differential equations: Let $ E \subset \mathbb{R}^{n}$ be an open set containing the origin  and  $f:E \to \mathbb{R}^{n}$ is continuously differentiable mappings such that $f(0)=0$
	\begin{align}\label{eq-ddeq1}
		\frac{dX(t)}{dt}&  = f(X(t)) ,~~
		X(0)  =x_{0}.
	\end{align}
	Assume that  the solution of the system \eqref{eq-ddeq1}  
	exists for every $t \geq 0$,  $\forall x_{0} \in E$. 
	\begin{itemize}
		\item Then origin is said to be  a {\em stable equilibrium point} of the system \eqref{eq-ddeq1} if  for every $\epsilon>0$ there exist $\delta >0$ such that $X(t,x_{0}) \in B(0, \epsilon)$ for every $x_{0} \in B(0, \delta)$ for every $t \geq 0$. 
		\smallskip
		
		\item The origin is said to be {\em globally attracting equilibrilum point or globally asymptotically stable equilibrium point} if $E=\rea^n$ , the origin is stable equilibrium point such that $\lim_{t\to \infty} X(t,x_{0}) =0$ for all $x_{0} \in \rea^n$. A vector field is said to be {\em globally asymptotically stable} if the origin is a globally asymptotically stable equilibrium point.
	\end{itemize}
	
	\noindent
	Any holomorphic vector field $F$ on $\cn$ can also be seen as a holomorphic map from $F:\cn \to \cn$. For such a vector field $F$, we denote $DF(a)$ by the derivative matrix viewing it as a map. 
	\begin{remark}
Let for $A \in M_{n}(\cplx)$,
 $m(A):=\inf_{\|z\|=1}\rl \langle Az, z\rangle$.	The condition in \Cref{thm-hamada2}, becomes $k_+(A)<2m(A)$.
 As we had discussed earlier the differential equation that we consider in this case is: 
	\begin{align}
		\frac{dX}{dt}&= AX(t)\label{eq-lin}\\
		X(0)&=z. \notag
	\end{align}	
	The differential equation that Hamada \cite{Hamada15} considered is the equation introduced by Gurganus \cite{Gur}	: 
	\begin{align}
		\frac{dX}{dt}&= - AX(t) \label{eq-linneg}\\
		X(0)&=z.\notag
	\end{align}	
%	If a domain $\Om\subset\cn$ is spirallike with respect to $A$ if and only if $\Om$ is spirallike with respect to $-A$ in the sense of Hamada (considering the differential equation as \eqref{eq-linneg}). 
The Runge property does not get affected by this change of definition. But, the approximation of the biholomorphic map of $\Om$ by automorphisms of $\cn$ get affected very much by this change. Hamada's condition $k_+(A)<2m(A)$ (the condition in Result~\ref{thm-hamada2}) does not remain invariant under the above change of definition. In general $m(A)$ is not a similarity invariant. Hence, the \Cref{thm-linear case} really enlarges the class of domains for which every biholomorphism with Runge image can be approximated by elements of $\text{Aut}(\cn)$ locally uniformly on $\Om$, even if the vector field is linear (see \Cref{ex-outside} for an explicit example).
\end{remark}
	
	Next, we state a result due to Arosio-Bracci-Hamada-Kohr \cite{ABHK2013} which gives the existence of a solution of Loewner PDE with values in a complex manifold on a complete hyperbolic domain. A manifold is said to be completely hyperbolic if it forms a complete metric space with respect to the Kobayashi distance (see \cite{Kobaya} for details.)
	
	\begin{result}\cite[Theorem 5.2]{ABHK2013}\label{T: Abstruct approach to Loe.}  Let $M$ be a complete hyperbolic complex manifold of dimension $n$. Let $G :M \times \mathbb{R}^{+} \to TM$ be a Herglotz vector field of order $d$. Then there exists a  Loewner chain $f_{t}:M \to N$ of order $d$ which solves the Loewner PDE 
		\begin{align} \label{E:LPDE}
			\frac{\partial f_{t}}{\partial t}(z)&=-df_{t}(z)(G(z,t)) ~~a.e.~ t \geq 0~~\forall z \in M ,
		\end{align}  
		where $N =\cup_{t \geq 0}f_{t}(M)$ is a complex manifold of dimension $n$ and any other solution to \eqref{E:LPDE} with values in a complex manifold $Q$ is of the form $(\Lambda \circ f_{t})$ where $\Lambda : N \to Q$ is holomorphic.
	\end{result}
	We now state a result by Thai-Duc \cite{Dudoc} which gives a characterization of the complete hyperbolic Hartogs domain. This plays a vital role in \Cref{T:Example of Domain}.
	\begin{result}\label{T:Do duc thai} \cite{Dudoc}
		Let $ X$ be a complex space and  $\phi$ a plurisubharmonic function on
		$X$. Assume that for every $x \in X$, there exists an open neighbourhood $U$ of $x$ in
		$X$ and a sequence $\{h_j\}$ of holomorphic functions on $ U$ and a sequence $ \{c_{j}\}$ of real
		numbers in the interval $(0, 1)$ such that the sequence $\{c_{j} \log |h_{j}|\}$ converges uniformly
		on compact subsets of $U$ to the function $\phi$. Then the Hartogs domain $\Omega_{\phi}(X)=\{(x,z) \in X\times \mathbb{C}: | z|<e^{-\phi(x)}\}$ is
		complete hyperbolic if and only if $X$ is complete hyperbolic and $\phi$ is continuous
		on $ X$.
	\end{result}

	Next, we state a result due to Hamada \cite[Theorem 5.2]{Hamada15} which describes a condition under which a pair of complex manifolds is a Runge pair.
	\begin{result}\label{T:rungepair}\cite{Hamada15}
		Let $M$ be  a complete hyperbolic Stein Manifold and $Q$ be a complex manifold of the same dimension. Let $(f_{t}:M \to Q)$ be a Loewner chain of order $d \in [1 ,\infty]$. Then $(f_{s_{1}}(M),f_{s_{2}}(M))$ is a Runge pair for $0 \leq s_{1}<s_{2}$.	
	\end{result}

 The next result due to Lloyd \cite[Theorem 3]{Lloyd} will be used to prove \Cref{L:composconver}.
 \begin{result}\cite[Theorem 3]{Lloyd}\label{R:Lloyd}
     Let $D$ be a bounded domain in $\cn$. Suppose
$f$ and $g$ are two continuous maps from $\overline{D}$ to $\cn$ which are holomorphic on $D$  such that
$$\|g(z)\| < \| f (z)\|~ \forall
z \in \partial D.$$
Then $f$ and $f + g$ have the same number of zeros in $D$ counting multiplicities,
where $\|\cdot\|$ is the standard norm in $\cn$.
 \end{result}
	
\noindent
	We need the following lemma  for the approximation result for the linear case.
	\begin{result} \cite[Page-311]{SimDav}\label{L:exp}
		 For every $\epsilon>0$ there exist $\rho(\epsilon), m(\epsilon)>0$, such that 
		 \begin{align*}
		 	\|e^{-tA}z\| &\leq \rho(\epsilon)e^{-(k_{-}(A)-\epsilon)t}\|z\|\\
		 	\|e^{tA}z\| &\leq m(\epsilon)e^{(k_{+}(A)+\epsilon)t}\|z\|
		 \end{align*}
	for all $t \geq 0$, and for all $z \in \cn$. 
\end{result}

 %Lemma about Runge domains
 We now state and prove a lemma that will be useful for our proofs.
 \begin{lemma}\label{lem-Runge-composition}
  Let $\Om$ be a domain in $\cn$ and $\vphi:\Om\to\cn$ be a biholomorphic map with $\vphi(\Om)$ Runge. Assume that $\mu:\Om\to\cn$ is a map such that $\mu(\Om)$ is Runge and $\mu(\Om)\subset\Om$. Then $\lambda\circ\vphi\circ\mu(\Om)$ is a Runge domain for every $\lambda\in Aut(\cn)$.
 \end{lemma}

 \begin{proof}
  It is enough to prove that $\vphi\circ\mu(\Om)$ is Runge. Let $f\in\hol(\vphi\circ\mu(\Om))$ and $K$ be a compact subset of $\vphi\circ\mu(\Om)$. Hence, $\vphi^{-1}(K)$ is a compact subset of $\mu(\Om)$. Clearly, $f\circ\vphi\in\hol(\mu(\Om))$. Since $\mu(\Om)$ is Runge, there exists a sequence of polynomials $\{p_n\}$ such that $p_n\to f\circ\vphi$ uniformly on $\vphi^{-1}(K)$. Hence, $p_n\circ\vphi^{-1}\to f$ uniformly on K. Since $\vphi^{-1}:\vphi(\Om)\to\Om$ is a holomorphic and $\vphi(\Om)$ is Runge, there exists a sequence $\{q_j\}$ of polynomial maps from $\cplx^n\to\cn$ such that $q_j\to \vphi^{-1}$ uniformly on $K$. By a diagonal-type process, we see that $f$ can be approximated by polynomials uniformly on $K$. Hence, $\vphi\circ\mu(\Om)$ is Runge.
 \end{proof}
 
We now state a result from \cite[Page-110]{Forstbook} that will be useful in the context of approximation by volume preserving automorphisms.
 Recall that, for holomorphic vector field $F =\sum_{j=1}^{n}F_{j}\frac{\partial}{\partial z_{j}}\in \mathfrak{X}_{\mathcal{O}}(\cn)$  and standard holomorphic $(n,0)$
form $\omega=\,dz_{1} \wedge \,dz_{2} \wedge \cdots \wedge \,dz_{n}$, $$\text{div}_{\om}F=\frac{\partial F_{1}}{\partial z_{1}}+\frac{\partial F_{2}}{\partial z_{2}}+\cdots +\frac{\partial F_{n}}{\partial z_{n}}.$$ 
\begin{lemma}\label{vloume preserve flow}
 	Let $\Om \subset \cn$ be an open set and $F: \Om \to \mathbb{C}^{n}$ be a complete $\smoo^{1}$- vector field with flow $X(t,z)$. Then $$\frac{d}{dt}\text{det} D_{z}(X(t,z))=\text{div}(F(X(t,z)))\cdot \text{det}~ D_{z}(X(t,z)).$$
 \end{lemma}
% SPIRALLIKE DOMAIN WRT ASYMTOTIC STABLE VF IS RUNGE
\section{Runge property}\label{sec-Runge}
Before proving the main result of this section we state and prove a proposition that will be useful in our proof of \Cref{thm-spirallikeRunge1}.

\begin{proposition}\label{P: global asymptotic property}
		Let $F \in \mathfrak{X}_{\mathcal{O}}{(\mathbb{C}^{n})}$ be a complete vector field such that the origin is a globally attracting fixed point and $\Omega$ be a positive time flow invariant domain containing the origin. Then, for every  compact subset $K \subset \mathbb{C}^{n}$, there exists a real number $M_{K}>0$ such that $X(t,z) \in \Omega$ for all $t >M_{K}$ and for all $z \in K$.
	\end{proposition}
	\begin{proof}
		Let $r>0$ such that $B(0,r) \subset \Om$.  Since the origin is the  globally attracting fixed point of  $F$, there exists $\delta>0$ such that $X(t,z) \in B(0,r)$ $\forall t \geq 0$ and  $\forall z \in B(0, \delta)$.
  We also have $\lim_{t \to \infty}X(t,z) \to 0$ for every  $z \in \cn$. Therefore, for every $z \in K$,  there exist $t_{z}>0$, such that $X(t_{z},z) \in B(0, \frac{\delta}{3})$. Hence, from choice of $\delta>0$, we have 
  $$
  X(t,z) \in B(0,r) \subseteq \Omega,\;\; \forall t \geq  t_{z}.
  $$
  Let $K \subset \cn$ be a compact subset. For every $z \in K$ there exist $r_{z}, t_{z}>0$ such that $X(t,w) \in B(0, \frac{\delta}{3})$, for all $w \in B(z,r_{z})$ and $t\geq t_{z}$. Since $K$ is compact, hence, there exists $m \in \mathbb{N}$ and $\{z_{1},z_{2}, \ldots ,z_{m}\} \subset K$ such that $K \subset \cup_{j=1}^{m}B(z_{j},r_{z_{j}})$. Let $M_{K}=\max \{t_{z_{1}},t_{z_{2}}, \ldots ,t_{z_{m}}\}$. Therefore, $X(t,z) \in B(0, r)$, for all $z \in K$ , for all $t>M_{K}$.

\end{proof}
	\smallskip
 
	We now present the proof of \Cref{thm-spirallikeRunge1}.
	
	\begin{proof}[Proof of Theorem~\ref{thm-spirallikeRunge1}]
		Let $f \in  \mathcal{O}(\Omega)$ and $K \subset \Omega$ be a given  compact subset and 
  $\mu$ is a  complex Borel measure on $K$.  It is enough to show that if, for every $p \in \mathbb{C}[z_1,z_2, \ldots, z_n]$,
  $$\int_{K}p(z_1,z_2, \ldots ,z_n)\;\mathrm{d}\mu(z) =0  $$ 
  then  $$ \int_{K}f(z)\;\mathrm{d}\mu(z)=0.$$

 Consider 
  $$
  V_{K}:=\{e^{-\tau}:\tau \in \mathbb{C}, X(\tau,z) \in \Omega, \forall z \in K\} \cup \{0\} \subset \mathbb{C}. 
  $$
  Since $F$ is a $\mathbb{R}$-complete vector field, hence, using \cite[Corollary 2.2]{Forst1996},  we get that    $F$ is also $\mathbb{C}$-complete. Therefore, $X(\tau,z)$ is well defined for every $\tau \in \mathbb{C}$. We now show that $V_{K}$ is open and starshaped with respect to the origin.
		If $w \in V_{K}$, then $w = e^{-\tau_{1}}$ for some $\tau_{1} \in \mathbb{C}$ with $X(\tau_{1}, z) \in \Omega$ for all $z \in K$. For any $t \in (0,1)$ we have 
  $$
			tw =e^{\ln{t}}\cdot e^{-\tau_{1}}
			=e^{-(-\ln{t}+\tau_{1})}.
$$
	Clearly, $-\ln{t} >0$. Since $\Omega $ is an invariant domain under the positive time flows of the vector field $F$, hence, for every $z \in K$, it follows that
  $$X(-\ln{t}+\tau_{1},z)=X_{-\ln{t}}((X(\tau_1,z))) \in \Omega.$$ 
  Consequently, $tw \in V_{K}$ for all $t \in (0,1)$.  Clearly, $tw \in \Omega$ for $t=0,1$. Therefore, $V_{K}$ is starshaped. Hence, it is simply connected. Now we show that $V_{K}$ is an open subset. For that, we first show that  
  $S=\{\tau \in \mathbb{C}|~~  \forall ~ z \in K,~X(\tau,z) \in \Omega\}$
  is an open subset of $\mathbb{C}$. Let $\tau_{1} \in S$. By  definition of $S$ we get that $X(\tau_{1},z) \in \Om$ for every  $z \in K$. Let $\epsilon_{z}>0$ such that $B(X(\tau_{1},z),\epsilon_{z}) \subset \Om$. Since $X$ is a continuous map, hence, there exists $\delta_{z,\tau_{1}}>0$  such that $X(\tau,w) \in B(X(\tau_{1},z),\epsilon_{z})$ whenever $|\tau-\tau_{1}|<\delta_{z,\tau_{1}}$ and $\|w-z\| <\delta_{z,\tau_{1}}$. Since $K$ is compact, hence, we get that $K \subset \cup_{i=1}^{m}B(z_{i},\delta_{z_{i},\tau_{1}})$ for some $z_{1},z_{2},\ldots ,z_{m}\in K$. Therefore, choosing $0<\delta_{1}<\min\{\delta_{z_{1},\tau_{1}},\delta_{z_{2},\tau_{1}},\ldots ,\delta_{z_{m},\tau_{1}}\}$ we obtain that $B(\tau_{1},\delta_{1}) \subset S$. By open mapping theorem, we get that all the points of the form $e^{-\tau}$ are interior points of $V_{K}$. 
  We now show that there exists $r>0$ such that $B(0,r) \subset V_{K}$.
		Now for every  non zero complex number $\tau \in \mathbb{C}$,  we have $\tau=e^{-(\ln{\frac{1}{|\tau|}}-i\arg{\tau})},$ where $-\pi <\arg{\tau} \leq \pi$. Hence, we get that  $X(\ln(|\frac{1}{\tau}|-i \arg{\tau})),z)=X_{\ln(|\frac{1}{\tau}|)}(X(-i \arg{\tau},z))$. Clearly, the set $K'=\{X(-it ,z): -\pi \leq t\leq \pi,~z\in K\}$ is compact.  Since $\ln{|\dfrac{1}{\tau}|} \to \infty$ as $\tau \to 0$, hence, we infer from \Cref{P: global asymptotic property}, that  there exists $r>0$ such that $X(\ln{\frac{1}{|\tau|}}, z') \in \Omega$, whenever $|\tau|<r$ and $z' \in K'$. We also obtain that  $X(-i\arg{\tau},z) \in K'$ for all $z \in K$. Hence, for any $\tau \in B(0,r)$  and $z \in K$, we have $X(\ln{\frac{1}{|\tau|}},X(-i\arg{\tau},z)) \in \Omega$. Therefore, $V_{K}$ is an open subset.  Let $\widetilde{V}_{K}=V_{K}\setminus (-\infty,0]$. For each $\tau \in \widetilde{V}_{K}$, define $f_{\tau}:K \to \mathbb{C} $  by $f_{\tau}(z)=f(X(-\ln{\tau},z))$. 
  %We first check the the map $f_{\tau}$ is  well defined. If $\tau \in V_{K}$ so $\tau=e^{-\tau_{1}}$ for some $\tau_{1}=a_1+ib_1 \in \mathbb{C}$ with for all $z \in K$ follows that $X(\tau_1,z) \in \Omega$. Now \begin{align*}
			%	\ln(\tau) &= \ln(|\tau|)+i \arg(\tau) \\
			%	& =\ln(|e^{-\tau_{1}}|)+i \arg(e^{-\tau_{1}})\\
			%	&=-a_1-i b_1\\
			%	&=-\tau_{1}
			%\end{align*} 
			%So $X(-\ln(\tau),z)=X(\tau_{1},z) \in \Omega$. 
			Here $f(X(-\ln{\tau},z))$ is continuously differentiable with respect to the variable  $\tau$ for each $z \in K$. Hence, the function $\phi: \widetilde{V}_{K} \to \mathbb{C}$ defined by 
   $$\phi(\tau)=\int_{K}f_{\tau}(z)\;\mathrm{d}\mu(z)
   $$ 
   is holomorphic on $\widetilde{{V}}_{K}$. Since $K$ is compact in $\cn$, there exist  $R>0$ such that $K \subset \overline{B(0,R)}$. In view of Proposition~\ref{P: global asymptotic property},  we get  that $\exists~ \delta>0 $ such that $X(-\ln{\tau},y) \in \Omega$ for all $\tau \in \{\tau \in \mathbb{C}:|\tau|<\delta\}\cap \widetilde{{V}}_{K}$ and for all $y \in \overline{B(0, R)}$. Therefore,  $f(X(-\ln{\tau},z)) \in \mathcal{O}(B(0,R))$ for every  $ \tau \in \{|t|<\delta\} \cap \widetilde{{V}}_{K} $. Hence, there exists a sequence of polynomial $p_{n}$ such that $p_{n}$ uniformly converges to $f_\tau$ on every compact subset of $B(0,R)$, particularly on $K$. Therefore, for every $\tau$ in $V_{K}$ with $|\tau|<\delta$,  it follows that 
   $$
   \lim_{n \to \infty}\int_{K}p_{n}(z)\;\mathrm{d}\mu(z)= \int_{K}f((X(-\ln{\tau},z)))\;\mathrm{d}\mu(z).
   $$
   By assumption $\int_{K}p_{n}(X(-\ln{\tau},z))\;\mathrm{d}\mu(z)=0$. Hence, for  all $\tau \in V_{K}$ with $|\tau| <\delta$,  we obtain that $$\int_{K}f((X(-\ln{\tau},z)))\;\mathrm{d}\mu(z)=0.$$  Hence, $\phi$ vanishes identically on an open ball in $\widetilde{V}_{K}$. Since $\phi$ is a holomorphic map hence, it vanishes identically on $\widetilde{V}_{K}$. Therefore, $\phi(1)\equiv0$ and this implies that   $\int_{K}f(z)\;\mathrm{d}\mu(z)=0$. Hence, $\Omega$ is Runge.
		\end{proof}

	%SECTION non LINEAR VECTOR FIELDS
	
	\section{Approximation of biholomorphisms}\label{sec-approx}
We begin this section with the following lemma which will be used in the proof of \Cref{thm-apprxges}. This lemma says under the conditions of \Cref{thm-apprxges} the asymptotic behavior of the flows of the vector field are like exponential up to a biholomorphism of the domain $\Om$.	

 \begin{lemma}\label{L:composconver}
Let $V \in \mathfrak{X}_{\mathcal{O}}{(\mathbb{C}^{n})}$, $n \geq 2$, be a complete holomorphic  vector field with $V(0)=0$ and $DV(0)=A$. Assume that $ \Om  \subseteq \cn$ contains the origin and is invariant under the positive time flows of $V$.  Suppose that  there exist real numbers $\gamma, \alpha>0$ such that 
   \begin{itemize}
  \item [(i)]  $\|X(t,z)\|\leq \gamma e^{-\alpha t}\|z\|$ for all $t \geq 0$ and for all $z \in \Om$, where $X(t,z)$ is the flow of  the vector field $V$.  
  
  \item [(ii)]
$2\alpha+k_{-}(A)>0.$
\end{itemize}
Then there exists a biholomorphic map $F$ on $\Om$ into  $\cn$ such that   $e^{-tA}\circ X_{t}(z)$ converges to $F(z)$ as $t \to \infty$ uniformly on compact subsets of $\Om$.
 \end{lemma}
 \begin{proof}
Fix a compact set $K \subset \Om$. We will show that for every $\eps>0$, there exists $N_{\eps,K} \in \mathbb{N}$ such that
$\|e^{-t_{1}A}X_{t_{1}}(z)-e^{-t_{2}A}X_{t_{2}}(z)\|<\eps$ for every $t_{1}, t_{2}>N_{\eps, K}$  and for all $z \in K$. This ensures the uniform convergence of $\{e^{-tA}\circ X_{t}\}$ on $K$ as $t \to \infty$.
Since  $V$ is a holomorphic vector field on $\cn$, we can write
$V(z)=Az+R(z)$, where $R(z)=O(\|z\|^2)$ as $z \to 0$ i.e., there exist $ M, \delta>0$ such that
$\|R(z)\|\leq M\|z\|^2$ for all $z \in B(0, \delta)$. Since $V$ is globally asymptotically stable vector field on $\cn$, for every $z\in\cn$
$\|X(t,z)\| \to 0$ as $t\to\infty$.  We note that the convergence is uniform on every compact subset of $\cn$. Hence, there exists $\wtil{N_{K}}>0$ such that 
\begin{align}\label{E:nconv2}
\|R(X(t,z))\|&<M\|X(t,z)\|^2, \forall t >\wtil{N_{K}} , \forall z \in K.
\end{align}
Define $f_t:\cn\to\cn$ by 
$f_{t}(z):=e^{-tA}X_{t}(z)$.  Clearly, $f_t\in Aut(\cn)$ for all $t>0$. The family is also smooth in $t$-variable.
Using the mean value theorem in the $t$ variable,
we obtain that
\begin{align}\label{E:nconv}
 \|f_{t}(z)-f_{s}(z)\|&\leq \left\|\frac{d f_{t}(z)}{\,dt}\bigg|_{t=s+a_{s}}\right\||t-s|,   
\end{align}
where $s\leq s+a_{s}\leq t $.
We now find an estimate for $\left\|\frac{d f_{t}(z)}{\,dt}\right\|$.
\begin{align}
\frac{d f_{t}(z)}{\,dt}
&=-Ae^{-tA}(X(t,z)+e^{-tA}(A(X(t,z)))+e^{-tA}(R(X(t,z))) \notag\\
&=e^{-tA}R(X(t,z)),\notag\\
\end{align}
Hence, 
\begin{align}\label{E:nconv1}
 \|\frac{d f_{t}(z)}{\,dt}\bigg|_{s+a_{s}}\|&=\|e^{-(s+a_{s})A}R(X((s+a_{s}),z))\| .
\end{align}
By assumption, choose $\sigma>0$ such that $2\alpha+k_{-}(A)-\sigma>0$. 
In view of \Cref{L:exp},  we deduce that
 \begin{align}\label{E:nconv3}
\|\frac{d f_{t}(z)}{\,dt}\bigg|_{s+a_{s}}\|& \leq 
\rho(\eps_{1})e^{-(s+a_{s})(k_{-}(A)-\sigma)}\|R(X((s+a_{s}),z))\|.
\end{align} 
\noindent
Using \eqref{E:nconv2} and \eqref{E:nconv3}, we get that 
  \begin{align}\label{E:nconv4}
\|\frac{d f_{t}(z)}{\,dt}\bigg|_{s+a_{s}}\|& \leq 
M\rho(\sigma)e^{-(s+a_{s})(k_{-}(A)-\sigma)}\|(X((s+a_{s}),z)\|^2\;\; \forall s>\wtil{N_{K}}\;\; \forall z \in K.
\end{align} 
Using Assumption~(i), equations \eqref{E:nconv}, \eqref{E:nconv2} and \eqref{E:nconv4} we obtain that
\begin{equation}\label{E:nconv5}
 \|f_{t}(z)-f_{s}(z)\|\leq M'e^{-Cs}|t-s|\;\; \forall t>s> \wtil{N_K}\;\; \forall z \in K,  
\end{equation}
where $C=2\alpha+k_{-}(A)-\sigma>0$ and $M':=\sup_{z \in K}M\gamma^2\rho(\eps_{1})\|z\|^2$.
%{\color{red} DONE TILL HERE}
Since there is an exponential decay on the right hand side of \eqref{E:nconv5}, it follows that 
for every $\eps>0$, there exists $N_{\eps,K} \in \mathbb{N}$ such that
$\|e^{-t_{1}A}X_{t_{1}}(z)-e^{-t_{2}A}X_{t_{2}}(z)\|<\eps$ for every $t_{1}, t_{2}>N_{\eps, K}$  and for all $z \in K$.
This proves that every subsequence of  $\{f_{t}\}$ forms a Cauchy sequence on every compact subset of $\Om$.
Hence, there exists $F\in\hol(\Om, \cn)$ such that $\lim_{t \to \infty}e^{-tA}\circ X_{t}(z)=F(z)$ and the convergence is uniform on every compact subsets of $\Om$.

Next, we show that $F$ is a biholomorphic map from $\Om$ onto its range.
We have $f_t=e^{-tA}\circ X_{t} \in \text{Aut}(\cn)$ and $D(e^{-tA}X_{t})(0)=e^{-tA}DX_{t}(0)=e^{-tA}e^{tA}=I_{n} $ for all $t\geq 0$. Since $ Df_{t}(0) \to DF(0)$ as $t \to \infty$, hence, we obtain that $DF(0)=I_{n}$. Since $\text{det}(D(e^{-tA}X_{t})(z)) \neq 0$ for all $z \in \cn$, therefore,  by Hurwitz theorem in several variables, we conclude that $\text{det}(DF(z))$ is either non-zero on $\Om$ or identically zero on $\Om$. But, $\text{det}(DF(0))=1$ implies that $\det(DF(z)) \neq 0$ for all $z \in \Om$. Therefore, by invoking inverse function theorem, we obtain that $F$ is a local biholomorphism.
We need to show that $F$ is injective. Suppose not. Then, there exist two distinct points $z_{1}, z_{2}$ in $\Om$ such that $F(z_{1})=F(z_{2})$. Since $F$ is locally injective, hence, there exists $r>0$ such that $F$ is injective on $B(z_1,r)$. Shrinking $r>0$, if needed, we assume that $z_{2} \notin \overline{B(z_1,r)}$. Define $\widetilde{F}_{t}:=f_{t}-f_{t}(z_{2})$
 and $\widetilde{F}:=F-F(z_{2})$. Clearly, $\widetilde{F}^{-1}(0) \cap \partial B(z_1,r) =\emptyset$. 
 Choose $\alpha>0$ such that $\inf_{\bdy B(z_1,r)}|\widetilde{F}|>\alpha$.
 
 Now consider $g_t:=\widetilde{F}_{t}-\widetilde{F}$. We have
 $$
     \sup_{\bdy B(z_1,r)}|g_t(z)|\leq \sup_{\bdy B(z_1,r)}\left ( |f_t(z)-F(z)|+|f_t(z_2)-F(z_2)|\right). 
 $$
 Hence, there exists $N_0\in\natu$ such that $\sup_{\bdy B(z_1,r)}|g_{t}(z)|\leq \al$ for all $t>N_0$. Hence, we obtain that 
 $$
 \sup_{\bdy B(z_1,r)}|g_t(z)|\leq \al < |\widetilde{F}(z)| \;\;\forall z\in\bdy B(z_1,r).
 $$
Therefore, we infer from \Cref{R:Lloyd}, that $\widetilde{F}$ and $\widetilde{F_t}$ has same number of zeros on $B(z_1,r)$ for all $t>N_0$. This contradicts the fact that $f_{t}$ are injective. Therefore, $F:\Om \to \cn$ is an injective holomorphic map.
\end{proof}
\smallskip

		We now present the proof of the main theorem of this paper Theorem~\ref{thm-apprxges}

		\begin{proof}[Proof of \Cref{thm-apprxges}]

 Here $V$ is a complete holomorphic vector field, $X\colon \mathbb{R} \times \cn \to \cn$ is the flow of the vector field $V$. Let $\Phi: \Om \to \Phi(\Om)$ be a biholomorphism such that $\Phi(\Om)$ is Runge.
 Define a map 
 $H\colon[0,1] \times \Omega \to \mathbb{C}^{n}$ 
	by \begin{align}\label{E:isononl}
	H(t,z):=
	\begin{cases}
		e^{c\ln{t}A}(\Phi(X(-c\ln{t},z))), ~\text{if}~ t \in (0,1]\\
		F(z), ~\text{if}~ t=0,\\
	\end{cases}
	\end{align} 
where $F : \Om \to \cn$ is a biholomorphism onto its image is chosen in such way  $e^{c\ln{t}A}\circ X_{-c\ln{t}}(z)$ converges to the map $ F(z)$ as $t \to 0^{+}$ uniformly over every compact subset of $\Om$. In view of \Cref{L:composconver}, we get the existence of such a biholomorphism $F\colon \Omega \to \mathbb{C}^n$. Clearly, the biholomorphism $F$ is a uniform limit of a sequence of automorphisms of $\cn$ on the compact subsets of $\Om$. By assumption, there exists $\epsilon_{1}>0$ such that   $2\alpha+k_{-}(A)-\epsilon_{1}>0$. We choose a real number $c>\frac{1}{2\alpha+k_{-}(A)-\epsilon_{1}}>0$ and fix it.  We will first show that $H$ is a $\smoo^{1}$-map and for each fixed $t \in [0,1]$, $H_{t}(\Om)$ is Runge. Then, we will use the \Cref{res-Forstneric-Rosay}.  It is enough to show that $H$ is $\smoo^{1}$ at the points of the form $(0,z)$. Since any translation map and an invertible linear map is particularly an automorphism of $\cn$, hence, without loss of generality, we assume that $\Phi(0)=0$ and $D\Phi(0)=I_{n}$. Let $K \subset \Om$ be a compact set.
We now divide the proof into four steps.

\noindent

	%In view of \Cref{T:Forstenric-1994} we conclude the theorem.
	\smallskip
	
	\noindent{\bf Step I:}{\bfseries \boldmath ~Showing that  $H$ is continuous  at $(0,z)$.}\\
Expanding the power series expansion of the biholomorphism $\Phi: \Omega \rightarrow \mathbb{C}^n$ around the origin, we obtain  that there exist  constants $M_{1}, \delta>0$  such that $\Phi$ can be expressed as:
\begin{align}\label{E:ncon1}
    \Phi(z)& = z + R_{1}(z) \quad \text{for all} ~ z \in B(0, \delta),
\end{align}
where 
\begin{align}\label{E:n-con1}
 \|R_{1}(z)\| \leq M_{1}\|z\|^2 ~  \text{for all}~ z \in B(0, \delta).  
\end{align}

\noindent
Clearly, $-c\ln{t} \to \infty$ as $t \to 0^{+}$. Since, by the assumptions $(i)$ and $(ii)$, the flow of the vector field $V$ has uniform exponential decay on the domain $\Om$, hence,  there exists $\delta_{K}>0$ such that $X(-c\ln{t},z) \in B(0, \delta)$, for all $z \in K$ and for all $t \in (0, \delta_{K})$.  
 Hence, by \eqref{E:n-con1}, we obtain 
 \begin{align}\label{E:n-con2}
   \|R_{1}(X(-c\ln{t},z))\|&\leq M_{1}\|X(-c\ln{t},z)\|^2  ~\forall z \in K, ~\forall t \in (0, \delta_{K}).  
 \end{align}
Again, using the uniform exponential decay in assumption $(ii)$, we obtain that
\begin{align}\label{E:n-con3}
   \|X(-c\ln{t},z)\|& \leq \gamma e^{\alpha c\ln{t}}\|z\|~ \forall t>0 .
\end{align}
\noindent
Hence, from \eqref{E:n-con2} and \eqref{E:n-con3}, we get that  
\begin{align}\label{E:ncon2}
   \|R_{1}(X(-c\ln{t},z))\|  &\leq M_{1}\gamma
    ^2t^{2c\alpha}\|z\|^2 ~ \forall z \in K,~ \forall t \in (0, \delta_{K}) .
    \end{align}
\noindent
By using \Cref{L:exp} and \eqref{E:ncon2}, we deduce that for all $z \in K$ and for all $t \in (0, \delta_{K})$ the following holds:
\begin{align}\label{E:ncon3}
   \|e^{c\ln{t}A}(R_{1}(X(-c\ln{t},z)))\|& \leq \rho(\eps_{1})e^{c\ln{t}(k_{-}(A)-\eps_{1})}\|R_{1}(X(-c\ln{t},z))\| \notag\\
   &\leq M_{1}\rho(\eps_{1})\gamma
    ^2e^{c\ln{t}(k_{-}(A)-\eps_{1})} t^{2c\alpha}\|z\|^2 \notag\\
    %&= M_{1}\rho(\eps_{1})\gamma
    %^2 t^{c(k_{-}(A)+2\alpha-\eps_{1})}\|z\|^2\notag\\
    &=M't^{c(k_{-}(A)+2\alpha-\eps_{1})}.
\end{align}
Here $M':=\sup_{z \in K}M_{1}\rho(\eps_{1})\gamma
    ^2 \|z\|^2$.
Since  from \eqref{C:conditi1} we have  $k_{-}(A)+2\alpha-\eps_{1}>0$, hence, from \eqref{E:ncon3}, we conclude that $e^{c\ln{t}A}(R_{1}(X(-c\ln{t},z))) \to 0$ as $t \to 0^{+}$, uniformly $K$.
\noindent
Now for any $(t,z) \in (0, \delta_{K})\times K$ we obtain
\begin{align}\label{E:ncon5}
H(t,z)&=  e^{c\ln{t}A}(\Phi(X(-c\ln{t},z)))\notag\\
%&=e^{c\ln{t}A}\big(X(-c\ln{t},z)+R_{1}(X(-c\ln{t},z))\big) \notag\\  
&=e^{c\ln{t}A}(X(-c\ln{t},z))+e^{c\ln{t}A}\big(R_{1}(X(-c\ln{t},z))\big).
\end{align}
\noindent
Therefore,
$$\lim_{t \to 0^{+}}H(t,z)=\lim_{t \to 0^{+}}e^{c\ln{t}A}(X(-c\ln{t},z)). $$
\noindent
Since, by \Cref{L:composconver}, we obtain that $\lim_{t \to 0^{+}}e^{c\ln{t}A}\circ (X(-c\ln{t},z))=F(z)$ and the convergence is uniform on every compact subset of $\Om$, therefore, $\lim_{t \to 0^{+}}H(t,z)=F(z)=H(0,z)$. Hence, $H$ is continuous at $(0,z)$.

\medskip

\noindent{\bf Step II:}{\bfseries \boldmath ~Showing that  $H$ is differentiable  at $(0,z)$.}

\noindent
We want to show that $DH(0,z)=DF(z)$ for every $z \in \Om$. For every $t \in (0, \delta_{K})$ , $z \in K$ , and for $h$ in a small enough neighborhood $U_0$ of the origin  we get that
\begin{align*}
& \|H(t,z+h)-H(0,z)-DF(z)h\|\\
&\qquad\qquad\qquad\qquad\qquad=  \|e^{c\ln{t}A}(\Phi(X(-c\ln{t},z+h)))-F(z)-DF(z)h\|. 
\end{align*}
Using the expression of the map $\Phi$ from \eqref{E:ncon1}, we get that
 \begin{align}\label{E:n-diff1}
 & \|H(t,z+h)-H(0,z)-DF(z)h\|\notag \\
 &= \|e^{c\ln{t}A}(X(-c\ln{t},z)+R_{1}(X(-c\ln{t},z)))-F(z)-DF(z)h\|\notag\\
  &\leq \|e^{c\ln{t}A}(X(-c\ln{t},z+h))-F(z+h)\|\notag\\
  & +\|e^{c\ln{t}A}\big(R_{1}(X(-c\ln{t},z+h)\big)+\|F(z+h)-F(z)-DF(z)h\|. 
 \end{align}
Hence, from \eqref{E:n-diff1}, we obtain, for $(t,z) \in (0, \delta_{K}) \times K$ and   $h\in U_0,\; h \neq 0$, that
\begin{align}\label{E:ndiff1}
& \frac{\|H(t,z+h)-H(0,z)-DF(z)h\|}{\|(t,h)\|}\notag\\
&\leq \frac{\|e^{c\ln{t}A}(X(-c\ln{t},z+h))-F(z+h)\|}{\|(t,h)\|}+\frac{\|e^{c\ln{t}A}\big(R_{1}(X(-c\ln{t},z+h)\big)\|}{\|(t,h)\|}\notag\\
 &+\frac{\|F(z+h)-F(z)-DF(z)h\|}{\|(t,h)\|}\notag\\
 &\leq \frac{\|e^{c\ln{t}A}(X(-c\ln{t},z+h))-F(z+h)\|}{t}+\frac{\|e^{c\ln{t}A}\big(R_{1}(X(-c\ln{t},z+h)\big)\|}{t}\notag\\
 &+\frac{\|F(z+h)-F(z)-DF(z)h\|}{\|h\|}.
  \end{align}
\noindent
Next, we show that the three summands that appeared on the right-hand side of the \eqref{E:ndiff1}, tend to zero as $\|(t,h)\| \to (0,0)$ with $t>0$.
Clearly, 
\begin{align}\label{E:ndiff6}
    \frac{\|F(z+h)-F(z)-DF(z)h\|}{\|h\|} \to 0~~ \text{as} ~~\|h\|~~ \to 0.
\end{align} 
\noindent
By similar computation as \eqref{E:ncon3}, we conclude that there exists $M''>0$, such that for every  $(t,h) \in (0, 1)\times \cn$ with  $\|(t,h)\|$ is small enough, the following holds:  
\begin{align}\label{E:ndiff2}
\frac{\|e^{c\ln{t}A}\big(R_{1}(X(-c\ln{t},z+h)\big)\|}{t}&\leq M'' t^{c(2\alpha+k_{-}(A)-\eps_{1})-1},~~\forall z\in K.
\end{align}
\noindent
From the choice of $c>0$, it follows that the right-hand side of the \eqref{E:ndiff2} tends to $0$ as $\|(t,h)\| \to 0$ with $t>0$. 
We now deal with the first summand of \eqref{E:ndiff1}.
\noindent
We now view the vector field $V$ as an entire mapping such that $V(0)=0$, $DV(0)=A$. Hence, we can write 
\begin{align}\label{E:n-diff2}
    V(z)&=Az+R(z). 
\end{align}
\noindent
Define a map $S\colon [0,1] \times \Omega \to \cn $ by 
$$
S(t,z)=
\begin{cases}
	e^{c\ln{t}A}\circ X_{-c\ln{t}}(z)~\quad\text{if}~ t \in (0,1] \\
	F(z) ~\quad\text{if}~ t=0\\
\end{cases}.
 $$ 
The first summand in \eqref{E:ndiff1} is 
\begin{align}\label{E:n-diff3}
 \frac{S(t,z+h)-S(0,z+h)}{t} .
\end{align}
 It follows, from \Cref{L:composconver}, that the map $t \to S(t, z+h)$  is continuous on $[0,1]$. It is also differentiable on $(0,1)$. Hence, by using the mean value theorem in \eqref{E:n-diff3}, we obtain that 
\begin{align}\label{E:n-diff4}
 \frac{\|S(t,z+h)-S(0,z+h)\|}{t}&\leq \bigg\|\frac{dS(t,z) }{dt}\bigg|_{t=\tau}\bigg\| .
\end{align}
Here we do some computation.
For $t \in (0,1)$, we have 
\begin{align}\label{E:ndiff7}
    \frac{dS(t,z) }{dt}&=\frac{d}{dt}(e^{c\ln{t}A}\circ X_{-c\ln{t}}(z)) \notag\\
    &=\frac{c}{t}Ae^{c\ln{t}A}\big(X(-c\ln{t},z)\big)+e^{c\ln{t}A}\bigg(\frac{d X(-c\ln{t},z)}{dt}\bigg)\notag \\
    &=\frac{c}{t}Ae^{c\ln{t}A}\big(X(-c\ln{t},z)\big)-\frac{c}{t}e^{c\ln{t}A}(V(X(-c\ln{t},z))). \notag\\
    \end{align}
\noindent
Since we have $V(z)=A(z)+R(z)$, hence, from \eqref{E:ndiff7} we get 
  \begin{align}\label{E:ndiff3}
  \frac{dS(t,z) }{dt} %&= \frac{c}{t}Ae^{c\ln{t}A}\big(X(-c\ln{t},z)\big)-\frac{c}{t}e^{c\ln{t}A}(A(X(-c\ln{t},z)))-\frac{c}{t}Ae^{c\ln{t}A}\bigg(R(X(-c\ln{t},z))\bigg)  \notag \\
  &=-\frac{c}{t}Ae^{c\ln{t}A}\big(R(X(-c\ln{t},z))\big).
  \end{align}
From \eqref{E:n-diff3} and \eqref{E:ndiff3} we deduce that 

\begin{align*}
\frac{\|S(t,z+h)-S(0,z+h)\|}{t}& \leq\bigg \|\frac{dS(t,z+h) }{dt}\bigg|_{t=\tau}\bigg\|,\notag
\end{align*}
where $0<\tau<t$. Therefore, we obtain

\begin{align}\label{E:ndiff4}
\frac{\|S(t,z+h)-S(0,z+h)\|}{t}&\leq \frac{c}{\tau}\bigg\|Ae^{c\ln{\tau}A}\big(R(X(-c\ln{\tau},z+h))\big)\bigg\|.
\end{align}
\noindent
There is a neighborhood 
$U$ of the origin such that   $\|R(z)\|\leq M\|z\|^2$ for all $z \in U$. Since $X(-c\ln{t},z) \to 0$ as $t \to 0^{+}$ uniformly on every compact subset of $\Om$. Hence, from \Cref{L:exp} and the equation \eqref{E:ndiff4}, we obtain that 
\begin{align}\label{E:ndiff5}
\frac{\|S(t,z+h)-S(0,z+h)\|}{t}&\leq M\gamma^2\rho(\eps_{1}) t^{c(2\alpha+k_{-}(A)-\eps_{1})-1}\|z+h\|^2.
\end{align}
\noindent
By  the choice of $c>0$, it follows from  \eqref{E:ndiff5} that  $$\frac{\|e^{c\ln{t}A}(X(-c\ln{t},z+h))-F(z+h)\|}{t} \to 0~ \text{as}~ t \to 0^{+}~ \text{uniformly on  $K$.}$$
\noindent
Therefore, we infer from \eqref{E:ndiff5}, \eqref{E:ndiff2} and \eqref{E:ndiff6} that left-hand side of \eqref{E:ndiff1} tends to $0$ as $\|(t,h)\| \to 0$ uniformly on $K$. Consequently, it proves that the map $H$ is differentiable at $(0,z)$. 
  
\medskip
	
	\noindent{\bf Step III:}{\bfseries \boldmath ~Showing that  $H$ is continuously differentiable  at $(0,z)$.}	
\smallskip

\noindent We show that $\frac{dH(t,z) }{dt} \to 0$ uniformly over $K$ as $t \to 0^{+}.$
We now compute the derivative. 
\begin{align}\label{E:ncodiff}
 \frac{dH(t,z) }{dt}&= \frac{d}{dt}(e^{c\ln{t}A}(\Phi(X(-c\ln{t},z)))) \notag \\
% &=\frac{c}{t}Ae^{c\ln{t}A}(\Phi(X(-c\ln{t},z)))-\frac{c}{t}e^{c\ln{t}A}\big(D\Phi(X(-c\ln{t},z))V(X(-c\ln{t},z))\big)\notag\\
 &=\frac{c}{t}e^{c\ln{t}A}\big(A\Phi(X(-c\ln{t},z))-D\Phi(X(-c\ln{t},z))V(X(-c\ln{t},z))\big)\notag \\
 &=\frac{c}{t}e^{c\ln{t}A}\big(\Psi(X(-c\ln{t},z))\big).
\end{align}
Here $\Psi(z)=A(\Phi(z))-D\Phi(z)V(z)$. Clearly, $\Psi(0)=0$. We prove that $D\Psi(0)=0$. We deduce the following for all non-zero $h \in \cn$ 
\begin{align}\label{E:ncodiff1}
 \|\Psi(h)\|&=\|A(\Phi(h))-D\Phi(h)V(h)\|\notag
    \\
    %&=\|A(\Phi(h))-D\Phi(h)(Ah+R(h))\|\notag\\
    %&=\|A(\Phi(h))-Ah+A(h)-D\Phi(h)Ah-D\Phi(h)R(h))\|\notag\\
    &\leq \|A(\Phi(h))-Ah\|+\|A(h)-D\Phi(h)Ah\|+\|D\Phi(h)R(h))\|\notag\\
    &\leq \|Ah+A(R_{1}(h))-Ah\|+\|A(h)-D\Phi(h)Ah\|+\|D\Phi(h)R(h))\|\notag\\
    &\leq\|A\|\|(R_{1}(h))\|+\|I-D\Phi(h)\|\|Ah\|+\|D\Phi(h)R(h))\|.
\end{align}
\noindent 
We have $\|R_{1}(h)\|\leq M_{1}\|h\|^2$ and $\|R(h)\|\leq M\|h\|^2$  for all $h \in U \cap B(0, \delta^{'})$. Also, $D\Phi(h) \to I$ as $h \to 0$. Hence, from \eqref{E:ncodiff1}, we get that $\frac{\|\Psi(h)\|}{\|h\|}\to 0$ as $\|h\| \to 0.$ This proves that $D\Psi(0)=0$. Therefore, we conclude that there exist  $M_{2}, \delta'>0$ such that  $$\|\Psi(z)\|\leq M_{2}\|z\|^2,~~ \forall z \in B(0, \delta').$$ 
We note that  $X(-c\ln{t},z) \to 0$ as $t \to 0^{+}$ uniformly on $K$. Hence,  by invoking \Cref{L:exp}, we conclude that there exists $\delta'_{K}>0$ such that   
 
 \begin{align}
 \bigg\|\frac{dH(t,z) }{dt}\bigg\|& \leq c\rho(\eps)M_{2}\gamma^2 t^{c(2\alpha+k_{-}(A)-\eps_{1})-1}\|z\|^2 ~~\forall t \in (0, \delta_{K}'),~  \forall z \in K 
\end{align}
It follows from the choice of $\eps_{1}, c>0$ that $\bigg\|\frac{dH(t,z) }{dt}\bigg\| \to 0$ as $t \to 0^{+}$ uniformly on $K$. 
Since $H_{t}$ is holomorphic on $\Om$, hence $DH_{t}(z) \to DH_{0}(z)=DF(z)$ uniformly on $K$ as $t \to 0^{+}$. This proves that $H$ is $\smoo^{1}$-smooth map  at $(0,z)$.
 \medskip
 
\noindent{\bf Step IV:}{\bfseries \boldmath ~Showing that $H_{t}(\Om)$ is Runge for $t \in [0,1]$.}
\smallskip

\noindent    Since both $X_{-c\ln t}$  and $X_{c\ln t}$ are automorphisms of $\cn$,  it follows from Lemma~\ref{lem-Runge-composition}, that $H_t(\Omega)$ is Runge. Hence, we are done.
\end{proof} 

 \begin{proof}[Proof of \Cref{thm-pro}]
 Define a vector field $V$ on $\cplx^{m+n}$ by 
 $$
 V(z,w):=(V_{1}(z),V_{2}(w)) \quad (z,w) \in \mathbb{C}^{n} \times  \mathbb{C}^{m}.
 $$
	Let $X_{1}(t,z)$ and $X_{2}(t,w)$ be the  flows of the vector field $V_{1}$ and $V_{2}$ respectively.
	Then $\theta(t, (z,w))=(X_{1}(t,z),X_{2}(t,w))$ is the flow of the vector field $V\in \mathfrak{X}_{\mathcal{O}}(\mathbb{C}^{n+m})$.  Clearly, if $(z,w) \in \Om_{1} \times \Om_{2}$, then $\theta_{t}(z,w) \in \Om_{1} \times \Om_{2}$ for every $t \geq 0$. Hence, $\Om_{1} \times \Om_{2}$ is invariant under the positive time flow of the  vector field $V$.
	
From the assumption we have that 	
\begin{align*}\label{E:prod}
2\alpha_{1}+k_{-}(A)&>0 \notag\\
  2\alpha_{2}+k_{-}(B)&>0. 
  \end{align*}
Here $A=DV_{1}(0)$ and $B=DV_{2}(0)$.
  
  Therefore, we obtain
  \begin{align}
	\|X_{1}(t,z)\|^2+\|X_{2}(t,w)\|^2& \leq \gamma_{1}^{2} e^{-2\alpha_{1}t}\|z\|^2+\gamma_{2}^{2} e^{-2\alpha_{2}t}\|w\|^2\notag\\
	\|\theta(t,(z,w))\|^2&\leq \tilde{\gamma}^{2}e^{-2\alpha t}\|(z,w)\|^2\notag \\
 \|\theta(t,(z,w))\|&\leq \tilde{\gamma}e^{-\alpha t}\|(z,w)\|,
	\end{align}
where $\tilde{\gamma}:=\max \{\gamma_{1}, \gamma_{2}\}$ and $\alpha:=\min\{\alpha_{1}, \alpha_{2}\}$.
 From the assumption we have $\alpha+\min\{\frac{k_{-}(A)}{2}, \frac{k_{-}(B)}{2}\}>0$. 

 Therefore,  the rest of the proof follows from   \Cref{thm-apprxges}.
	\end{proof}
  
 \begin{proof}[Proof of \Cref{thm-Volsympapprox}]
 
 {\sf(i) } Suppose that $\Phi : \Om \to \Phi(\Om)$ is a biholomorphism with $\text{det}D\Phi(z) \equiv 1$ and $\Phi(\Om)$ is Runge domain. According to our assumption, $div_{\omega}(V)$ is a non-zero  constant.  Let $div_{\omega}(V)=a \neq 0$. By   \Cref{vloume preserve flow}, we get that
			\begin{align}
			 \frac{d}{dt}\det D_{z}(X(t,z))&=a \det~ D_{z}(X(t,z)),~\text{det}D_{z}(X(0,z))=1,\notag
			\end{align}
where $D_{z}(X(t,z))$ denotes the derivative matrix of the automorphism $X_{t} \in \text{Aut}(\cn)$ with respect to $z$. Solving the above differential equation, we obtain that
 \begin{align}\label{E:vol2}
	\text{det} D_{z}(X(t,z))&=e^{at},  \;\; \forall z \in \cn,\;\; \forall t \geq 0.
\end{align}
  Since $\det D_{z}X_{t}(0)=e^{tr(A)t}$, where $A=DV(0)$, hence, from \eqref{E:vol2}, we get that  $a=tr(A)$. We also get from \eqref{E:isononl}
	\begin{align}
		D_{z}H_{t}(z)=e^{c\ln{t}A}\circ D_{z}(\Phi(X(-\ln{t},z))) \circ D_{z}X_{-\ln{t}}(z).\notag
	%\impl & \det D_{z}H_{t}(z)	=\det(D_{z}(X_{\ln{t}})(\Phi(X(\ln{t},z)))))\det(D_{z}(\Phi(X(-\ln{t},z))))\det( D_{z}X_{-\ln{t}}(z)).\notag
 \end{align}
	Since $\text{det}(D\Phi(X(-c\ln{t},z)))\equiv 1$, hence, from \eqref{E:vol2} we get that
 $$
	\det D_{z}H_{t}(z)=e^{tr(A)\ln{t}-tr(A)\ln{t}}=1,~~\forall t \in(0,1).
 $$
 Since $H_{t}(z) \to F(z)$ as $t \to 0^{+}$, hence, $\text{det}H_{t}(z) \to \text{det}F(z)$ as $t\to 0^{+}. $ Therefore, $ \text{det}F(z) \equiv 1$ on $\Om$.
 \smallskip
 
 \noindent We now show that identity map $\iota_{\Omega}: \Om \to \Om$ is homotopic equivalent to a constant map. Define  $ \widetilde{H}: [0,1] \times \Om \to \cn$  by 
   $$
   \widetilde{H}(t,z)=
			\begin{cases}
				X(-\ln{t},z),~\text{if}~ t \in (0,1]\\
				0, ~\text{if}~ t=0.\\
			\end{cases}.
   $$
 The global exponential growth of the flow map provides the continuity of the map $\widetilde{H}(t,z)$ at the point $(0,z)$.   This proves that the domain $\Om$ is contractible to $0$. Hence, $H^{(n-1)}(\Om, \mathbb{C})$ vanishes identically. Since $\Om$ is also pseudoconvex,  therefore,
			by \Cref{res-Forstneric-Rosay}, we conclude that $H_{1}(z)=\Phi(z)$ can be approximated by $ \text{Aut}_{1}(\cn)$ uniformly on every compact subset of $\Om$.
   \medskip
   
{\sf (ii)} The flow of the symplectic vector field preserves the symplectic form. From our assumption $\Phi:\Om \to \Phi(\Om)$ is a symplectic holomorphic map. 
Hence, $H_{t}^{*}(\omega)=\omega$ for all $t \in (0,1]$. We have seen in the proof of {\sf{(i)}} that $\Om$ is contractible. Hence, $H^{(n-1)}(\Om, \mathbb{C})=0$. We get that $H_{t}(\Om)$ is Runge for every $t \in [0,1]$, from Step IV of the proof of \Cref{thm-apprxges}. Now using  \cite[Proposition 2.3]{symfor1995} we obtain our theorem.
 
  \end{proof}

		\section{Application in Loewner PDE}\label{sec-Application}
		In this section first we state a generalization of 
 \cite[Theorem 3.4]{ABFW13} and \cite[Theorem 4.5]{Hamada15}.  
	
  \begin{theorem}\label{T: union of Runge stein domain}
			Let $\Omega$ be a complex manifold of dimension $N \geq 2$. Let $\{\Omega_{j}\}_{j \in \mathbb{N}\cup \{0\}}$ be a sequence of open connected subset of $\Omega$ such that $\Omega_{j} \subset \Omega_{j+1}$ for every $j \in \mathbb{N}$ and $\Omega=\cup_{j \in \mathbb{N}\cup \{0\}}\Omega_j$. Assume that \begin{enumerate}
				\item [(i)]
				each pair $(\Omega_{j},\Omega_{j+1})$ is a Runge pair.
				\item [(ii)]
				each $\Omega_{j}$ is biholomorphic to a Stein domain 
    that is invariant under positive time flows of a vector field $V_{j} \in \mathfrak{X}_{\mathcal{O}}(\mathbb{C}^{n})$ satisfying the conditions of \Cref{thm-apprxges}.
			\end{enumerate}Then $\Omega$ is biholomorphic to a Runge and Stein domain in $\mathbb{C}^{n}$.
		\end{theorem}
	\noindent Since we have the approximation on certain domains satisfying conditions  \Cref{thm-apprxges}, \Cref{thm-linear case}, hence, the proof of the above theorem follows  exactly the same way as \cite[Theorem 3.4]{ABFW13}. Therefore, we omit the proof here.	
\smallskip
  
  We now give proof of \Cref{T:Loewner Pde}.

  \begin{proof}[Proof of Theorem~\ref{T:Loewner Pde}]
			Using \Cref{T: Abstruct approach to Loe.} we get that there exists a Loewner chain $g_{t}: \Omega \to N$ of order $d$ that solves the Loewner PDE 
			\begin{align}\label{E:LPDE1}
				\frac{\partial f_{t}}{\partial t}(z)&=-df_{t}(z)(G(z,t)), ~~a.e.\quad t \geq 0,~~\forall z \in \Omega,
			\end{align}
			where $N=\bigcup_{t \geq 0} g_{t}(\Om)$ is a complex manifold of dimension $n$. We also obtain that any other solution to equation \eqref{E:LPDE1} with values in another $n$ dimensional complex manifold $Q$ is of the form $\Lambda \circ g_{t}$, where $\Lambda : N \to Q$ is holomorphic. By using \Cref{T:rungepair}, $\left(g_{j}(\Omega),g_{j+1}(\Omega)\right)$ is a Runge pair for every $j \in \mathbb{N}$. Since  $g_{t}$ is a univalent holomorphic map,  $g_{t}(\Omega)$ is an open connected subset of $N$. Clearly, $g_{t}(\Om)$ is biholomorphic to $\Om$. By assumption, $\Om$ is invariant under positive time flows of $V$ such that $V$ satisfies the  \eqref{C:conditi1}. Therefore, $g_{j}(\Om)$ satisfies  the Condition~$(ii)$ of \Cref{T: union of Runge stein domain} for all $j \in \mathbb{N}$. Hence,  thanks to \Cref{T: union of Runge stein domain}, we get that  $N$ is biholomorphic to some Runge Stein domain  $G \subset \mathbb{C}^{n}$. Hence, it follows that there exists a biholomorphic map $F: N \to G$.  Let  $f_{t}=F \circ g_{t}$.  We have the following 
			\begin{align}
				\frac{\partial f_{t}}{\partial t}(z)&=DF(g_{t}(z))	\frac{\partial g_{t}}{\partial t}(z)\notag\\
				&=DF(g_{t}(z))(-dg_{t}(z)(G(z,t)))\notag\\
				%&=-DF(g_{t}(z))(dg_{t}(z)(G(z,t)))\notag\\
				&=-D(F(g_{t}(z)))G((z,t))\notag \\
				&=-D(f_{t}(z))(G(z,t)).\notag
			\end{align}  
			This implies that $f_{t}$ satisfies \eqref{E:LPDE1}. If $\tilde{g_{t}}(z)$ is another solution of the  (\ref{E:LPDE1}) with values in $\cn$, then we deduce, from Theorem~\ref{T: Abstruct approach to Loe.}, that there is a biholomorphism $\Lambda:N \to \cup_{t \geq 0}\tilde{g_{t}}(\Omega)$ such that $\tilde{g_{t}}=\Lambda \circ g_{t}=\Lambda \circ F^{-1}(f_{t})$.  Consider $\vphi=\Lambda \circ F^{-1} $. Therefore, any solution of \eqref{E:LPDE1}  with values in $\mathbb{C}^{n}$ is of the form $\vphi(f_{t})$ for some suitable biholomorphism $\vphi\colon G \to \cn$.
		\end{proof}

		% SECTION-EXAMPLES
\section{Examples}\label{sec-example}

In this section, we provide a some examples of domains that satisfy the conditions of our theorems. 
\begin{example}
Let $A=\text{Diag}(-\lambda_{1}, -\lambda_{2}, \cdots ,-\lambda_{n})$, such that $\lambda_{i} \in \mathbb{N}$, with, $\lambda_{1} \leq \lambda_{2} \cdots \leq \lambda_{n}$ and $2\lambda_{1}>\lambda_{n}$. Let $i, j \in \{1,2, \cdots ,n\}$ and $m_{i}, m_{j} \in \mathbb{N}$, such that $m_{i}\lambda_{i}=m_{j}\lambda_{j}$. Let $\Om:=\{z \in \cn: \big|z_{i}^{m_{i}}-z_{j}^{m_{j}}\big|<1\}$. Then every biholomorphism from $\Om$ onto a Runge domain can be approximated by elements of $\text{Aut}(\cn)$.
\end{example}	
\noindent {\bf Explanation.}
Let $z \in \Om$ and $w=e^{tA}z$. Then we obtain that 
\[
|w_{i}^{m_{i}}-w_{j}^{m_{j}}| =|e^{-m_{i}\lambda_{i}t}z_{i}-e^{-m_{j}\lambda_{j}t}z_{j}|=e^{-m_{i}\lambda_{i}}|z_{i}-z_{j}|<1.
\]
 Since $2\lambda_{1}>\lambda_{n}$, hence, it satisfies the condition of the  \Cref{thm-linear case}. Therefore, the conclusion is true.

The next example demonstrates a class of domains that satisfy the conditions of \Cref{thm-linear case} but do not satisy the conditions of \Cref{thm-hamada2}.
\begin{example}\label{ex-outside}
Let $A=\lambda I_{n}+\mathcal{N}$,  where $\mathcal{N}$ is a $n \times n$ matrix with ones in the first diagonal above the main diagonal and zeros elsewhere and $\lambda=\lambda_{1}+i\mu_{1}$, with $\lambda_{1}<0$. Let 
$\Om:=\{(z_{1}, z_{2}, ,\ldots ,z_{n}): \bigg|z_{n-1}-\frac{z_{n}}{\lambda_{1}}\ln{|z_{n}|}\bigg|<1\}$. Then every biholomorphism from $\Om$ onto a Runge domain can be approximated by elements of $\text{Aut}(\cn)$.
\end{example}

\noindent {\bf Explanation.}
Suppose that  $A=-\lambda I_{n}+\mathcal{N}$. Let $w=e^{At}z$. We get that $$w=e^{At}=e^{\lambda_{1}t}
\begin{bmatrix}
1 & t &  \frac{t^{2}}{2} & \cdots & &\frac{t^{n}}{n!}
\\
0& 1& t& \cdots &	 &\frac{t^{n-1}}{(n-1)!}\\
\vdots & \vdots & \vdots&\ddots &&\vdots\\
0&0&0&\cdots&  &t\\
0& 0&0&\cdots &&1 \\
\end{bmatrix}.$$
For all $z \in \Om$ and for all $t \geq 0$,  we obtain that 
\begin{align*}
	\bigg|w_{n-1}-\frac{w_{n}}{\lambda_{1}}\ln{|w_{n}|}\bigg|& =\bigg|z_{n-1}e^{\lambda t}+tz_{n}e^{\lambda t}-\frac{z_{n}e^{\lambda t}}{\lambda_{1}}(\ln{|z_{n}|}+t\lambda_{1})\bigg| \\
	&=\bigg|z_{n-1}e^{\lambda t}-e^{\lambda t}\frac{z_{n}}{\lambda_{1}}\ln{|z_{n}|}\bigg|\\
&=e^{\lambda_{1} t}	\bigg|z_{n-1}-\frac{z_{n}}{\lambda_{1}}\ln{|z_{n}|}\bigg|\\
&<1.
\end{align*}
 Clearly, the matrix $A$ satisfies 
$2k_{+}(A)=2\lambda_1<\lambda_1$. 
Therefore, from \Cref{thm-linear case} it follows that any biholomorphism of $\Om$, whose image is a Runge domain in $\cn$, can be approximated by elements of $\text{Aut}(\cn)$.
But the matrix $-A=-\lambda I_{n}-\mathcal{N}$ may not  satisfy the condition of \Cref{thm-hamada2}. Since for any $B \in M_{n}(\cplx)$ we have that $m(B):=\inf\{\rl \langle Bz, z\rangle: \|z\|=1\}=\min\{\rl (\lambda): \lambda \in \sigma (\frac{1}{2}(B+B^{*}))\}$, hence for $n=3$ $m(A)=-\lambda_{1}-\sqrt{2}$ and $k_{+}(A)=-\lambda_{1}$. Therefore, for $n=3$ the matrix satisfies the condition given in \Cref{thm-hamada2} if and only if $\lambda_{1}<-2\sqrt{2}$.

\begin{example}
Let $D \subset \cn$ be a domain such that $D$ is invariant under the positive time flow of the vector field $V \in  \mathfrak{X}_{\mathcal{O}}(\cn)$. Assume that $V$ satisfies the condition of \Cref{thm-apprxges}. Let $\Om =\{(z,w) \in \mathbb{C}^{n+1}:z \in D,~ \|w\|<1\} \subset \mathbb{C}^{n+1}$. Then $\Om$ is invariant under positive time flows of the vector field $\widetilde{V}=(V(z),-\beta w)$, where $\beta >0$ is chosen suitably so that the vector field  $\widetilde{V}$ satisfies the condition of \Cref{thm-apprxges}.
\end{example}

  \noindent {\bf Explanation.}
	The flow of the vector field $\tilde{V}$ is $\tilde {X}(t,(z,w))=(X(t,z), e^{-\beta t}w )$, where $X(t,z)$ denotes the flow of the vector field $V$.
Clearly, $\tilde{X}_{t}(\Om) \subset \Om$ for all $t>0$. Here $\tilde{X}_{t}$ is the flow of the vector field  $\widetilde{V}=(V(z),-\beta w)$ and $\beta >0$ will be chosen later. Let $DV(0)=A$.    Since the flow of the vector field $V$ satisfies the condition of the \Cref{thm-apprxges}, hence, for all $z=(z_{1}, z_{2}) \in D$, there exist $\gamma,  \alpha >0$, such that 
\begin{align}\label{E:exam}
	\|X(t,z)\| &\leq \gamma e^{-\alpha t}\|z\|,\;\;\forall z \in D\;\; \forall t>0,
 \end{align}
 and $\alpha >\frac{-k_{-}(A)}{2}$.  Choose $\beta>0$  such that $$\min\{\alpha, -k_{-}(A)\}>\beta >\frac{-k_{-}(A)}{2}.$$ For $(z,w) \in \Om$ and $t>0$, we have the following  estimate:  
\begin{align}
	\|\tilde{X}(t,(z,w))\|^2 & = \|X(t,z)\|^2+e^{-2\beta t}\|w\|^2 .\notag\\
	\intertext{Hence, from \eqref{E:exam} we get that  }
\|\tilde{X}(t,(z,w))\|^2	&\leq \tilde{\gamma}^2 e^{-2\beta t}(\|z\|^2+ \|w\|^2), \notag\\
\intertext{where 
 $\tilde{\gamma}^{2}:=\max\{1, \gamma^2\}.$ Therefore, for every $(z,w) \neq (0,0)$, we obtain that }
\|\tilde{X}(t,(z,w))\|	 &\leq \tilde{\gamma}e^{-\beta t}\|(z,w)\|. \notag
\end{align}
 Clearly, from the choice of $\beta>0$ we have  $k_{-}(D\tilde{V}(0))=\min\{k_{-}(A), -\beta\}=k_{-}(A)$. Again from the  choice of $\beta>0$, we have $2\beta+k_{-}(D\widetilde{V}(0))>0$.
Therefore, it satisfies Condition~\eqref{C:conditi1} of \Cref{thm-apprxges}.

		Next, we provide an example of a complete hyperbolic Hartogs domain which is invariant in the positive time flow of a  nonlinear vector field but not invariant  under the flow of any matrix. More precisely:
  
		\begin{example}\label{T:Example of Domain}
			
Consider the domain  $
  \Omega=\{(z_1,z_2)\in \mathbb{C}^{2}:~\rl(z_{1})<3,~|z_2|<e^{-\frac{1}{2}\rl(z_{1})})\} $. $\Om$ is a complete hyperbolic Hartogs domain which is invariant under positive time flows of a non-linear holomorphic vector field but not invariant under positive time flows of any linear holomorphic vector field.
\end{example}

  \noindent {\bf Explanation:}
We first show that $\Omega$ is invariant under positive time flows of the vector field $V(z_1,z_2)=(-2z_1,-3z_2+z_1z_2)$.
The flow of the vector field $V$ is 
   \[
   X(t,z)= \left(z_1e^{-2t},z_2e^{-3t}e^{\frac{z_1}{2}(1-e^{-2t})}\right), \;\;\forall t\in \rea.
   \]
	Let $(z_1,z_2) \in \Omega$ and $t \geq 0$. Let $w_{1}=z_{1}e^{-2t}$ and $w_2=z_2e^{-3t}e^{\frac{z_1}{2}(1-e^{-2t})}$. Then we have:
			\begin{align*}
				|w_2|&= |z_2|e^{-3t}||e^{\frac{z_1}{2}(1-e^{-2t})}|\\
				&\leq 	|z_2||e^{\frac{z_1}{2}(1-e^{-2t})}|.
    \end{align*}
				Since $(z_1,z_2)\in \Omega$, we get that
    \begin{align*}
				|w_{2}|	&\leq  e^{-\frac{1}{2}\rl(z_{1})} e^{\frac{1}{2}\rl(z_{1})(1-e^{-2t})}\\
				%&< e^{-\frac{1}{2}e^{-2t}\rl(z_1)}\\
				&\leq e^{-\frac{1}{2}\rl(z_{1}e^{-2t})}\\
				&< e^{-\frac{1}{2}\rl(w_{1})}.		
    \end{align*}
			It is clear that $\rl(w_{1})<3$. Therefore, we get  $(w_1,w_2)\in \Omega $ for all $t \geq 0$. 
   \smallskip
   
 \noindent Next, we claim that $\Om$ is complete hyperbolic. We use  \Cref{T:Do duc thai} to prove this lemma. Consider the sequence  $c_{j}=\frac{1}{2}$  and $h_{j}(z)=e^{z}$ for $j \in \mathbb{N}$. Then $\{c_{j}\log|{h_{j}}(z)|\}$ converges to $\frac{1}{2}\rl(z)$. Since upper half space $H:=\{z\in \mathbb{C}:\Im(z)>0\}$ is biholomorphic to a domain $G:=\{z\in \mathbb{C}:\rl(z)<3\} \subset \mathbb{C}$. Hence, $G$ is a complete hyperbolic domain. We now use \Cref{T:Do duc thai} to  conclude that   $\Omega$ is a complete hyperbolic domain. Next, we show that $\Om$ satisfies the conditions of \Cref{thm-apprxges}. We show that the flow of the vector field $V$ satisfies Condition~\eqref{C:conditi1}. Let $\zeta=(\zeta_{1}, \zeta_{2}) \in \Om$. We have the estimate of the flow on the domain $\Om$ 
			\begin{align*}
				\|X(t,\zeta)\|^{2}& =e^{-4t}\|\zeta_{1}\|^2+\|\zeta_{2}\|^2e^{-6t}e^{\rl(\zeta_{1})(1-e^{-2t})}.
    \end{align*}
    Here we have $\rl(\zeta_{1})<3$. We also get that there exists $M>0$ such that 
   $-2t+(1-e^{-2t})\rl{\zeta_1}<1$ for all $t >M$   . Let $s=\sup_{t \in [0, M]}-2t+3(1-e^{-2t})$ Therefore, we have 
   \begin{align*}
      \|X(t,\zeta)\|^2 &\leq \gamma^2 e^{-4t}\big(\|\zeta_{1}\|^2+\|\zeta_{2}\|^2 \big)\\
      \|X(t,\zeta)\|  &\leq  \gamma e^{-2t}\|\zeta\|,
   \end{align*}
where $\gamma^{2}:=\max\{1, s\}$.				

Here $\alpha=2$ and $k_{-}(DF(0))=-3$. Hence, $2\alpha+k_{-}(DF(0))>0$.  Hence, $\Om$ satisfies \eqref{C:conditi1}.	Therefore. in view of Theorem~\ref{thm-spirallikeRunge1}, any biholomorphism of $\Om$ can be approximated uniformly on compacts by Aut$(\CC)$.
\smallskip

\noindent It can be shown that $\Om$ is not invariant under positive time flows of any linear holomorphic vector field. It is clear that if all the eigenvalues of $A$ have positive real parts then a domain $\Om$ containing the origin  satisfies $e^{tA}\Om \subset \Om$ for all $t>0$ if and only if  $\Om=\mathbb{C}^{n}$.  For   given a matrix $A \in GL(n , \cplx)$, all of whose eigenvalues have negative real part, we construct an initial point $z=(z_{1}, z_{2}) \in \Om$  and $t>0$ (depending on $A$) such that $e^{tA}z \notin \Om$. We provide the main steps here.

\smallskip

\noindent
\textbf{Step I:}
\textbf{Showing that $\Om$ is not invariant under positive time flows of the holomorphic linear vector fields corresponding to any diagonal matrix:}  

\noindent
In this case, we may assume that $A=\begin{bmatrix}
						d_{1} & 0 \\
						0 & d_2 \\
			\end{bmatrix}$ with $d_{1}=a_{1}+ib_{1}$ and $d_{2}=a_{2}+ib_{2}$ and $a_{j}, b_{j} \in \mathbb{R}$ for $j=\{1, 2\}$. If $b_{1} \neq 0$, then we choose $(-2p, 1) \in \Om$ with $\frac{a_{2}\pi}{|b_{1}|}+p e^{\frac{a_{1}\pi}{|b_{1}|}}>0$ and $t=\frac{\pi}{|b_{1}|}$. If $b_{1}=0$ then the choice of the initial point and the time depends on the sign of $a_{1}$ as follows: If $a_{1}<0$, then we choose the initial point $z=(-2p,z_{2})$ with $e^{p}>z_{2}>e^{pe^{a_{1}}+|a_{2}|}$ and $p=\frac{2|d_{2}|}{1-e^{a_{1}}}$ and $t_{1}=1$. If $a_{1}>0$ then we take $z=(2,\frac{1}{e^{2}}) \in \Omega$ and $t>0$ with $e^{a_{1}t}+a_{2}t-2>0$.    
\medskip

\noindent
   \textbf{Step II:}
\textbf{Showing that $\Om$ is not invariant under positive time flows of the holomorphic linear vector fields corresponding to any matrix of the form $\begin{bmatrix}
    \lambda & 1\\
    0  & \lambda
\end{bmatrix}$:}  
Here we consider the matrix of the form
 $\begin{bmatrix}
    \lambda & 1\\
    0  & \lambda
\end{bmatrix}$, $\lambda=\lambda_{1}+i\lambda_{2}$. In case, $\lambda_{2} \neq 0$, we choose $z=(2p,-1)$ with $	\frac{\lambda_{1} \pi}{|\lambda_{2}|}-\frac{\pi e^{\frac{\pi \lambda_1}{|\lambda_{2}|}}}{2|\lambda_{2}|}+pe^{\frac{\lambda_{1} \pi}{|\lambda_{2}|}} >0 $ and $t=\frac{\pi}{|\lambda_{2}|}$. If $\lambda_{2}=0$ then the choice of the initial point and the time depends on the sign of $\lambda_{1}$. If $\lambda_{1}<0$ we take $(-2te^{-\lambda_{1}t},2e^{-\lambda_{1}t}) \in \Om$ such that $t>0$ and $2e^{t(k-e^{kt})}<1$. 

\medskip

\noindent
If $A$ is similar to a diagonal matrix or a Jordan block then also we can compute the initial point inside the domain and the positive exit time of the integral curve. Since  the computations are lengthy and of similar nature as  Step I and II above, we omit the computation here.

\smallskip

Next, we provide an example of a domain in $\cn$ similar to previous example. 
 \begin{example}
		Let $\Om =\{(z_{1}, \ldots ,z_{n}) \in \cn: \rl(z_{1})<3,~ |z_{j}|<e^{-\frac{1}{2}\rl(z_{1})}, 2\leq j\leq n\}$. Then $\Om$ satisfies the condition of \Cref{thm-apprxges} with respect to the vector field $V(z_{1},z_{2}, \ldots ,z_{n})=(-2z_{1},-3z_{2}+z_{1}z_{2},-3z_{3}+z_{3}z_{1}, \ldots , -3z_{n}+z_{n}z_{1})$. It can be proved, in similar lines to the proof of \Cref{T:Example of Domain}. 
	\end{example}

		\medskip
		
		\noindent {\bf Acknowledgements.} 
		Sanjoy Chatterjee is  supported by CSIR fellowship (File No-09/921(0283)/2019-EMR-I) and also would like to thank Golam Mostafa Mondal and Amar Deep Sarkar for several discussions. Sushil Gorai is partially supported by a Core Research Grant (CRG/2022/003560) of SERB, Dept. of Science and Technology, Govt. of India.

		\bibliographystyle{plain}
		%\bibliography{../biblio.bib}

		\begin{comment}

			%\subsection{Bibliography}
			%\bibliographystyle{siam}
			
		\end{comment}
	\end{document}